\documentclass[10pt,reqno]{amsart}
\usepackage[british]{babel}
\usepackage{amsmath,amsfonts,amssymb,amstext,amsbsy,amsopn,amsthm,amsxtra,amscd} 
\usepackage{graphicx}
\usepackage{pdfpages}
\usepackage{caption}
\usepackage{subcaption}
\usepackage{color}
\usepackage{hyperref}
\usepackage{enumerate}
\usepackage{rotating}

 \newcommand{\acknowledgment}[1]{\par\addvspace\baselineskip\noindent\textbf{Acknowledgment:}\enspace\ignorespaces#1}

\numberwithin{equation}{section}

\newtheorem{theorem}{Theorem}[section]
\newtheorem{proposition}[theorem]{Proposition}
\newtheorem{corollary}[theorem]{Corollary}
\newtheorem{lemma}[theorem]{Lemma}
\newtheorem{definition}[theorem]{Definition}
\newtheorem{remark}[theorem]{Remark}
\newtheorem{claim}[theorem]{Claim}
\newtheorem{notation}[theorem]{Notation}

\newcommand{\bb}[1]{\mathbb{#1}}

\DeclareMathOperator{\conv}{conv}
\DeclareMathOperator{\diam}{diam}
\DeclareMathOperator{\dist}{dist}

\newcommand{\sob}[1]{\hspace{0.5pt}\raisebox{-1.8pt}{$#1$}}

\newcommand{\lle}{\mathchoice
	{\rotatebox[origin=c]{-180}{$\ell$}}
	{\rotatebox[origin=c]{-180}{$\ell$}}
	{\rotatebox[origin=c]{-180}{\scalebox{.7}{$\ell$}}}
	{\rotatebox[origin=c]{-180}{\scalebox{.5}{$\ell$}}}
}

\title{An Alphabetical Approach to Nivat's Conjecture}
\author{Cleber F. Colle}
\address{UNICAMP - University of Campinas, 13083-859, Campinas, SP, Brazil}
\email{clebercolle@outlook.com}
\author{Eduardo Garibaldi}
\address{UNICAMP - University of Campinas, 13083-859, Campinas, SP, Brazil}
\email{garibaldi@ime.unicamp.br}


\begin{document}
\begin{abstract}
Since techniques used to address the Nivat's conjecture usually relies on Morse-Hedlund Theorem, an improved version of this classical result may mean a new step towards a proof for the conjecture. In this paper, consi- dering an alphabetical version of the Morse-Hedlund Theorem, we show that, for a configuration $\eta \in A^{\mathbb{Z}^2}$ that contains all letters of a given finite alphabet $A$, if its complexity with respect to a quasi-regular set $\mathcal{S} \subset \mathbb{Z}^2$ (a finite set whose convex hull on $\mathbb{R}^2$ is described by pairs of edges with identical size) is bounded from above by $\frac{1}{2}|\mathcal{S}|+|A|-1$, then $\eta$ is periodic.\\

\centering This version is based on the PhD thesis defended in 2017.
\end{abstract}
	
\maketitle 

\section{Introduction}
\label{sec1}

Fixed a finite alphabet $A$ (with at least two elements), for $n \in \bb{N}$, the $n$-comple- xity of an infinite sequence $\xi = (\xi_ {i})_ {i \in \bb{Z}} \in A^{\bb{Z}}$, denoted by $P_{\xi}(n)$, is defined to be the number of distinct words of the form $\xi_j\xi_{j+1} \cdots \xi_{j+n-1}$ appearing in $\xi$. In 1938, Morse and Hedlund \cite{hedlund} proved one of the most famous results in symbolic dynamics which establishes a connection between periodic sequences (sequences for which there is an integer $m \geq 1$ such that $\xi_{i+m} = \xi_{i}$ for all $i \in \bb{Z}$) and complexity. More specifically, they proved that $\xi \in A^{\bb{Z}}$ is periodic if, and only if, there exists $n \in$ $\bb{N}$ such that $P_{\xi}(n) \leq n$.

A natural extension of the complexity function to higher dimensions is obtained when we consider, instead of words, blocks of symbols. More precisely, the $n_1 \times \cdots \times n_d$-complexity of a configuration $\eta = (\eta_{g})_{g \in \bb{Z}^d} \in A^{\bb{Z}^d}$, denoted by $P_{\eta}(n_1, \ldots, n_d)$, is defined to be the number of distinct $n_1 \times \cdots \times n_d$ blocks of symbols appearing in $\eta$. Of course periodicity also has a natural higher dimensional generalization: $\eta \in A^{\bb{Z}^d}$ is periodic if there exists a vector $h \in (\bb{Z}^d)^*$, called period of $\eta$, such that $\eta_{g+h} = \eta_g$ for all $g \in \bb{Z}^d$. If $\eta \in A^{\bb{Z}^2}$ has two linearly independent periods, it is said to be doubly periodic. A configuration that is not periodic is said to be aperiodic.

The Nivat's Conjecture \cite{nivat} is the natural generalization of the Morse-Hedlund Theorem for the two-dimensional case.

\medbreak\noindent{{\bf Conjecture} (Nivat).}\hspace{1ex}\ignorespaces {\it For a configuration $\eta \in A^{\bb{Z}^2}$, if there exist integers $n,k \in$ $\bb{N}$ such that $P_{\eta}(n,k)$ $\leq nk$, then $\eta$ is periodic.}\medbreak

The first step towards the conjecture was given by Sander and Tijdeman \cite{tijdeman}: they showed that if $P_{\eta}(n,2) \leq 2n$ (or if $P_{\eta}(2,n) \leq 2n$) for some integer $n \in \bb{N}$, then $\eta \in A^{\bb{Z}^2}$ is periodic. Other weak forms of the Nivat's Conjecture were successively obtained in \cite{epifanio,quas,durand,van1,kari}. Moreover, Sander and Tijdeman \cite{sander} found counter-examples to the analogue of Nivat's Conjecture in higher dimensions, i.e., they showed that, for $d\geq 3$, there exist periodic configurations $\eta \in$ $\{0,1\}^{\bb{Z}^d}$ such that $P_{\eta}(n, \ldots,n) = 2n^{d-1}+1$.

The best result known so far was obtained by Bryna Kra and Van Cyr \cite{van}. Using the notion of expansive subspaces of $\bb{R}^2$ introduced by Boyle and Lind, they shed a new light towards a proof for Nivat's Conjecture by relating expansive subspaces to periodicity. In particular, they proved that if there exist integers $n,k \in \bb{N}$ such that $P_{\eta}(n,k) \leq \frac{1}{2}nk$, then $\eta \in A^{\bb{Z}^2}$ is periodic.

Our main result (Theorem \ref{main_result}) is an alphabetical improvement on Cyr and Kra's Theorem. Moreover, we consider the complexity function with respect to a more general class of sets, called quasi-regular sets  (see Definition \ref{q-regular}). In the particular case of blocks, we show that, for a configuration $\eta \in A^{\bb{Z}^2}$ that contains all letters of $A$, if there exist integers $n,k \in \bb{N}$ such that
\begin{equation}\label{eq_intro_mainresul}
P_{\eta}(n,k) \leq \frac{1}{2}nk+|A|-1 = \left(\frac{1}{2}+\frac{|A|-1}{nk}\right)nk,
\end{equation} 
where $|A|$ denotes the cardinality of the alphabet $A$, then $\eta$ is periodic. 

Here is an example of configuration that satisfies (\ref{eq_intro_mainresul}) but does not satisfy the condition of Cyr and Kra's Theorem. Let $A$ be the alphabet formed by the colours ``white" and ``black" and define $\eta \in A^{\bb{Z}^2}$ as $\eta_g := ``\textrm{black}"$ if $g = (a,a)+b(\sum_{i=6}^c i,0)$, where $a \in \bb{Z}$, $b \in \{-1,0,1\}$ and $c \geq 6$, and $\eta_g := ``\textrm{white}"$ otherwise (see Figure~$\ref{exemplo}$). Note that $P_{\eta}(n,k) = n+k$ when $n+k \leq 7$ and that from the symmetries of such configuration $$P_{\eta}(n,k) = n+k+\frac{1}{2}(n+k-7)(n+k-6)$$ when $n+k > 7$. It is easy to see that there are no integers $n,k \in \bb{N}$ such that $P_{\eta}(n,k) \leq \frac{1}{2}nk$. However, one has $P_{\eta}(3,4) = 7 = \frac{1}{2}\cdot 12+|A|-1$.
\begin{figure}[ht]
\centering
\def\svgwidth{11.5cm}
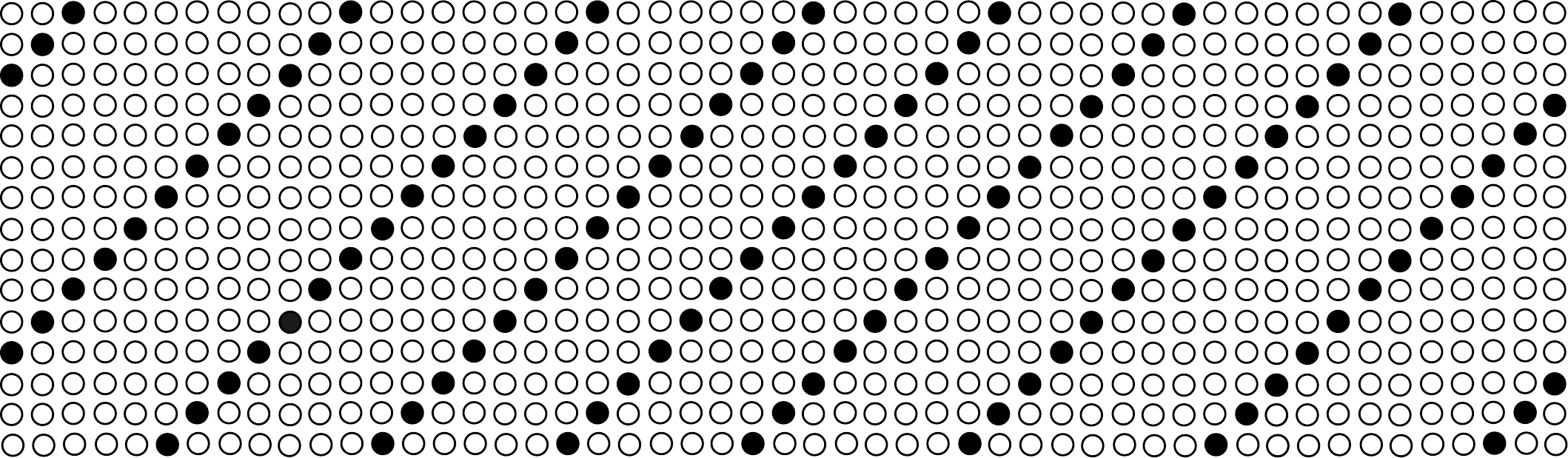
\caption{Representation of the configuration $\eta \in A^{\bb{Z}^2}$.}
\label{exemplo}
\end{figure}

\section{Initial concepts and main result}
\label{sec2}

Let $A$ be endowed with the discrete topology. It is well known that the configuration space $A^{\bb{Z}^d}$ equipped with the product topology is a metrizable compact space. For each $u \in \bb{Z}^d$, let $T^u : A^{\bb{Z}^d} \! \rightarrow A^{\bb{Z}^d}$ be the shift application, i.e., for $\eta = (\eta_{g})_{g \in \bb{Z}^d}$ $\in A^{\bb{Z}^d}$, the configuration $T^u\eta$ is defined by $(T^u\eta)_{g} := \eta_{g+u}$ for all $g \in \bb{Z}^d$. It is easy to see that, with respect to this topology, the $\bb{Z}^d$-action by shift applications $(T^u : u \in \bb{Z}^d)$ is continuous.  Let $Orb\,(\eta) := \{T^{u}\eta : u \in \bb{Z}^d\}$ denote the $\bb{Z}^d$-orbit of $\eta \in A^{\bb{Z}^d}$ and set $X_{\eta} := $ $\overline{Orb\,(\eta)}$, where the bar denotes the closure.

 Following Sander and Tijdeman \cite{sander}, for a nonempty set $\mathcal{S} \subset \bb{Z}^d$, the $\mathcal{S}$-complexity of $\eta \in A^{\bb{Z}^d}$, denoted by $P_{\eta}(\mathcal{S})$, is defined to be the number of distinct $\mathcal{S}$-con\-fig\-ur\-a\-tions of the form $(T^{u}\eta)|\sob{\mathcal{S}} \in A^{\mathcal{S}}$, where $u \in \bb{Z}^d$ and $\cdot |\sob{\mathcal{S}}$ means the restriction to the set $\mathcal{S}$. The set of all $\mathcal{S}$-configurations of $\eta \in A^{\bb{Z}^d}$ is denoted by $$L(\mathcal{S},\eta) := \left\{(T^{u}\eta)|\sob{\mathcal{S}} \in A^{\mathcal{S}} : u \in \bb{Z}^d\right\}.$$ Clearly $\mathcal{T} \subset \mathcal{S}$ implies $P_{\eta}(\mathcal{T}) \leq P_{\eta}(\mathcal{S})$. If $\mathcal{S} \subset \bb{Z}^d$, then $P_{\eta}(\mathcal{S}) = P_{T^u\eta}(\mathcal{S}+g)$ for all $u,g \in \bb{Z}^d$ and $P_{x}(\mathcal{S}) \leq P_{\eta}(\mathcal{S})$ for any $x \in X_{\eta}$. We remark that for a $n_1 \times \cdots \times n_d$ block based at the origin, i.e., $$R_{n_1, \ldots, n_d} := \left\{(t_1, \ldots, t_d) \in \bb{Z}^d : 0 \leq t_i < n_i \ \textrm{for every} \ 1 \leq i \leq d\right\},$$ the previous notion $P_{\eta}(n_1,\ldots,n_d)$ coincides with $P_{\eta}(R_{n_1,\ldots,n_d})$. 

A set $\mathcal{S} \subset \bb{Z}^2$ is called \emph{convex} if its convex hull in $\bb{R}^2$, denoted by $\conv(\mathcal{S})$, is closed and $\mathcal{S} =$ $\conv(\mathcal{S}) \cap \bb{Z}^2$. If $\mathcal{S} \subset \bb{Z}^2$ is a convex set, a point $g \in \mathcal{S}$ is a \emph{vertex} of $\mathcal{S}$ when $\mathcal{S} \backslash \{g\}$ is a convex subset, and a line segment $w$ contained at the boundary of $\conv(\mathcal{S})$ is an \emph{edge} of $\mathcal{S}$ if it is an edge of the convex polygon $\conv(\mathcal{S}) \subset \bb{R}^2$. Let $V(\mathcal{S})$ and $E(\mathcal{S})$ denote, respectively, the sets of vertices and edges of $\mathcal{S}$. 

Let $\mathcal{F}_{C}$ denote the family of convex, finite and nonempty subsets of $\bb{Z}^2$, and $\mathcal{F}^{V\!ol}_{C}$ de\-no\-te the subfamily of $\mathcal{F}_{C}$ whose convex hull has positive area.

If $\mathcal{S} \subset \bb{Z}^2$ is a convex set (possibly infinite) such that $\conv(\mathcal{S})$ has non-null area, our standard convention is that the boundary of $\conv(\mathcal{S})$ is positively oriented. With this convention, each edge $w \in E(\mathcal{S})$ 
inherits a natural orientation from the boundary of $\conv(\mathcal{S})$.

In the sequel, by an oriented object we mean an oriented line, an oriented line segment or a vector. Remember that two vectors are parallel if they have the same direction and antiparallel if they have opposite directions. Two oriented objects in $\bb{R}^2$ are said to be (anti)parallel if the adjacent vectors to their respective orientations are (anti)parallel.

\begin{definition}\label{q-regular}
We say that $\mathcal{S} \in \mathcal{F}^{V\!ol}_{C}$ is a quasi-regular set when, for every edge $w \in E(\mathcal{S})$, there is an edge $w' \in E(\mathcal{S})$ antiparallel to $w$ satisfying $|w' \cap \mathcal{S}| = |w \cap \mathcal{S}|$.
\end{definition}

We may now state our main result.

\begin{theorem}\label{main_result}
If  $\eta \in A^{\bb{Z}^2}$ contains all letters of the alphabet $A$ and there exists a quasi-regular set $\mathcal{S} \in$ $\mathcal{F}^{V\!ol}_{C}$ such that $P_{\eta}(\mathcal{S}) \leq \frac{1}{2}|\mathcal{S}|+|A|-1$, then $\eta$ is periodic.
\end{theorem}

The proof of Theorem \ref{main_result} will be done by contradiction. The global strategy consists  basically in showing the existence of an aperiodic accumulation configuration which is doubly periodic on a unbounded  region, and then in  arguing that the cardinality of subconfigurations arising in the boundary transition from doubly periodicity to aperiodicity would be greater than possible. We will first apply the Alphabetical Morse-Hedlund Theorem to certain strips, defined from special generating sets, called balanced sets (see Definitions~\ref{dfn_generating} and \ref{dfn_lp_balanceado}), to get periodic con- figurations, which will allow us, by an inductive argument, to construct such an accumulation point. 

The rest of the paper is organized as follows. The next section reviews key con- \linebreak cepts and results. In Section \ref{sec4}, with an alphabetical viewpoint, we obtain a wide range of propositions which will be useful for the proof of the main theorem. Rou- ghly speaking, these results connect one-sided nonexpansive directions (which will be precisely defined in Section \ref{sec3}) and periodicity. In Section \ref{sec5}, we define a special class of generating sets (sets that determine a larger region from a smaller region). Following methods highlighted by Cyr and Kra \cite{van}, in the Section \ref{sec6}, we provide bounds for the periods of any aperiodic configuration with an unbounded doubly periodic region (see Definition \ref{ll'maximalperconf}) and in the last section we prove Theorem \ref{main_result}. The results of this article are based on the Ph.D thesis of the first author written under the guidance of the second author.

\section{Fundamental notions and facts}
\label{sec3}

From now on, we will assume that the configuration $\eta \in A^{\bb{Z}^2}$ always contains all letters of the alphabet $A$, i.e., $P_{\eta}(\{g\}) = |A|$ for all $g \in \bb{Z}^2$.

\subsection{Alphabetical Morse-Hedlund Theorem}

We reproduce here an alphabetical version of the celebrated result of Morse and Hedlund  (see Theorem 7.4 in \cite{hedlund}).

Given a sequence $\xi = (\xi_t)_{t \in U} \in A^U$, where $U = \{a,a+1,\ldots\}$ and $a \in \bb{Z}$, the \emph{$n$-complexity} of $\xi$, denoted by $P_{\xi}(n)$, is defined to be the number of distinct words of the form $\xi_{t}\xi_{t+1} \cdots \xi_{t+n-1}$, where $\{t,t+1, \ldots, t+n-1\} \subset U$. Such sequence is said to be \emph{periodic} if there exists an integer $m \geq 1$ (called period) that satisfies $\xi_{i+m} = \xi_{i}$ for all $i \in U$.

\begin{theorem}[Alphabetical Morse-Hedlund Theorem]\label{AMHT}
Let $\xi = (\xi_i)_{i \in U} \in A^{U}$ be a sequence  that contains all letters of $A$, where $U = \bb{Z}$ or $U = \{a,a+1,\ldots\}$ for some $a \in \bb{Z}$. Suppose there exists $n_0 \in \bb{N}$ such that $P_{\xi}(n_{0}) \leq n'_0$, where $n'_0 := n_{0}+|A|-2$.
\begin{enumerate}
	\item[(i)] If $U = \{a,a+1,\ldots\}$, then the sequence $(\xi_{t})_{t \in U+n'_0} \in A^{U+n'_0}$ is periodic of period at most $n'_{0}$;
	\item[(ii)] If $U = \bb{Z}$, then the sequence $\xi \in A^{\bb{Z}}$ is periodic of period at most $n'_0$.
\end{enumerate}
\end{theorem}

As an immediate consequence, we have the next result.

\begin{corollary}\label{lem_per_convnulo}
Let $\eta \in A^{\bb{Z}^2}$ and suppose $P_{\eta}(\mathcal{S}) \leq$ $|\mathcal{S}|+|A|-2$ for some set $\mathcal{S} \in \mathcal{F}_{C}$. If $\conv(\mathcal{S})$ has null area, then $\eta$ is periodic.
\end{corollary}

\subsection{Expansive subdynamics}
Let $F$ be a subspace of $\bb{R}^d$. For each $g \in \bb{Z}^d$, de- note $\dist(g,F) := \inf \{\|g-u\| : u \in F\}$, where $\| \cdot \|$ is the Euclidean norm in $\bb{R}^d$. Given $t > 0$, the $t$-neighbourhood of $F$ is defined as $$F^{t} := \left\{g \in \bb{Z}^d : \dist(g,F) \leq t\right\}.$$

Let $X \subset A^{\bb{Z}^d}$ be a subshift (a closed subset which is invariant for shift applications). Following Boyle and Lind \cite{boyle}, we say that a subspace $F \subset \bb{R}^d$ is \emph{expansive} on $X$ if there exists $t>0$ such that $x|\sob{F^{t}} = y|\sob{F^{t}}$ implies $x = y$ for any $x,y \in X$. If a subspace fails to meet this condition, it is called a \emph{nonexpansive} subspace on $X$. Boyle and Lind \cite[Theorem~3.7]{boyle} showed that if $X \subset A^{\bb{Z}^d}$ is an infinite sub\-shift, then, for $0 \leq k < d$, there exists a $k$-dimensional subspace of $\bb{R}^d$ that is nonexpan- sive on $X$.

As an immediate corollary from Boyle and Lind's Theorem, we highlight the following result. 

\begin{corollary}\label{cor_2_periodic}
For $\eta \in A^{\bb{Z}^2}$, if every one-dimensional subspace of $\bb{R}^2$ is expansive on $X_{\eta}$, then $\eta$ is doubly periodic.
\end{corollary}

The next result allows us to conclude that every configuration with at least two nonexpansive one-di\-men\-sional subspaces on $X_{\eta}$ is not periodic. 

\begin{lemma}\label{pps_perio_expan}
If $\eta \in A^{\bb{Z}^2}$ is periodic of period $h \in (\bb{Z}^2)^*$, then every one-di\-men\-sional subspace $F \subset \bb{R}^2$ that does not contain $h$ is expansive on $X_{\eta}$.
\end{lemma}
\begin{proof}
Let $t>0$ be such that $\{-h,h\} \subset F^{t}$. For each $g \in \bb{Z}^2$, there exists an inte- ger $m \in \bb{Z}$ such that $g + mh \in F^{t}$. Since all con\-fi\-gu\-rations $x,y \in X_{\eta}$ are periodic of period $h \in (\bb{Z}^2)^*$, if $x|\sob{F^{t}} = y|\sob{F^{t}}$, then $x_{g} = x_{g + mh} = y_{g + mh} = y_{g}$.
\end{proof}

In the above lemma, note that there is no assumption about the expansiveness or nonexpansiveness of the line which contains the period $h$, so that doubly periodic configurations naturally satisfy its statement. Actually, if a configuration is doubly periodic, by applying Lemma \ref{pps_perio_expan} to linearly independent periods, we see that all its one-dimensional subspaces  are expansive.

A convex set $\mathcal{H} \subset \bb{Z}^2$ is said to be a \emph{half plane} if $\conv(\mathcal{H})$ has non-null area and $E(\mathcal{H})$ has only a single edge. In this case, the single edge $\ell \in E(\mathcal{H})$ is a line in $\bb{R}^2$. For a line $\ell \subset \bb{R}^2$, we also use $\ell$ to denote this line endowed of a given orientation, and $\lle$ to denote the same line endowed of the opposite orientation. We believe that, according to the context, the reader will easily realize if we refer to a line or to an oriented line.  For an oriented line $\ell \subset \bb{R}^2$ that contains at least one point of $\bb{Z}^2$, let $\mathcal{H}(\ell) \subset \bb{Z}^2$ denote the unique half plane for which $\ell$ is its single (positively) oriented edge. This means not only that $\ell \in E(\mathcal{H}(\ell))$ but also the orientation of the edge of the half plane $\mathcal{H}(\ell)$ agrees with the orientation of $\ell$.

\begin{notation}\label{not_nextline}
If $\ell$ is a rational oriented line (i.e., an oriented line with rational angular coefficient) that contains at least one point of $\bb{Z}^2$, let $\ell^{(-)} \subset \bb{R}^2$ denote the oriented line parallel to $\ell$ such that the half plane $\mathcal{H}(\ell^{(-)})$ is minimal (with respect to partial or\-der\-ing by inclusion) among all half planes that strictly contains $\mathcal{H}(\ell)$. Likewise, let $\ell^{(+)} \subset \bb{R}^2$ denote the oriented line parallel to $\ell$ such that the half plane $\mathcal{H}(\ell^{(+)})$ is maximal among all half planes that are strictly contained in $\mathcal{H}(\ell)$.
\end{notation}

We use $\bb{G}_1$ to denote the set of all lines through the origin in $\bb{R}^2$, i.e., the set of one-dimensional subspaces. In a slight abuse of notation, we also say that oriented lines belong to $\bb{G}_1$. If $\ell \in \bb{G}_1$ is a rational oriented line, let $\vec{v}_{\ell} \in (\bb{Z}^2)^*$ denote the non-null vector parallel to $\ell$ of minimal norm.

In the sequel, we restate a refined version of the classical notion of expansiveness called one-sided nonexpansiveness, and introduced by Cyr and Kra in \cite{van}.

\begin{definition}
Given $\eta \in A^{\bb{Z}^2}$,  we say that an oriented line $\ell \in \bb{G}_1$ is a one-sided expansive direction on $X_{\eta}$ if $x|\sob{\mathcal{H}(\ell)} = y|\sob{\mathcal{H}(\ell)}$ implies $x = y$ for any $x,y \in X_{\eta}$. If an oriented line $\ell \in \bb{G}_1$ fails to meet this condition, it is called a one-sided nonex- pansive direction on $X_{\eta}$. 
\end{definition}

As $X_{\eta}$ is a compact subshift of $A^{\bb{Z}^2}$, it is easy to see that $\ell \in \bb{G}_1$ is an expansive line on $X_{\eta}$ if, and only if, $\ell, \lle \in \bb{G}_1$ are one-sided expansive directions on $X_{\eta}$. More- over, in all the text, whenever we consider a nonexpansive line, this actually means a nonexpansive rational line, since any irrational line is expansive on $X_{\eta}$ if the configuration $\eta$ satisfies $P_{\eta}(\mathcal{U}) \leq |\mathcal{U}|+|A|-2$ for some $\mathcal{U} \in \mathcal{F}_{C}$ (see Remark~\ref{Obs_irracional_line}).

\subsection{Generating sets}
The notion of generating set, deeply developed in \cite{van}, underlines configurations that admit a unique extension on extreme points of a given convex set.

\begin{definition}\label{dfn_generating}
Let $\eta \in A^{\bb{Z}^2}$ and suppose $\mathcal{S} \subset \bb{Z}^2$ is a finite set. A point $g \in \mathcal{S}$ is said to be $\eta$-generated by $\mathcal{S}$ if $P_{\eta}(\mathcal{S}) = P_{\eta}(\mathcal{S} \backslash \{g\})$. A set $\mathcal{S} \in \mathcal{F} _{C}$ for which each vertex is $\eta$-generated is called an $\eta$-ge\-ne\-rating set. 
\end{definition}

Note that $g \in \mathcal{S}$ is $\eta$-generated by $\mathcal{S}$ if, and only if, for every $\gamma \in L(\mathcal{S}\backslash\{g\},\eta)$, there exists a unique $\gamma' \in L(\mathcal{S},\eta)$ such that $\gamma'|\sob{\mathcal{S}\backslash\{g\}} = \gamma$. 

\begin{remark}\label{obs_exist_generating}
If $P_{\eta}(\mathcal{U}) \leq|\mathcal{U}|+|A|-2$ for some set $\mathcal{U} \in \mathcal{F}_{C}$, then any convex set $\mathcal{S} \subset \mathcal{U}$ that is minimal among all convex sets $\mathcal{T} \subset \mathcal{U}$ fulfilling $P_{\eta}(\mathcal{T}) \leq |\mathcal{T}|+|A|-2$ is an $\eta$-generating set. The fact that $1+|A|-2 < |A| = P_{\eta}(\{g\})$ for all $g \in \bb{Z}^2$ en- sures that $\mathcal{S}$ has at least two points.
\end{remark}

\begin{definition}\label{retabase}
Given an oriented line $\ell \subset \bb{R}^2$ and a convex set $\mathcal{S} \subset \bb{Z}^2$, we use $\ell_{\scriptscriptstyle\mathcal{S}}$ to denote the oriented line $\ell' \subset \bb{R}^2$ parallel to $\ell$ such that $\mathcal{S} \subset \mathcal{H}(\ell')$ and $\ell' \cap \mathcal{S} \neq \emptyset$.
\end{definition}

Note that either $\ell_{\scriptscriptstyle\mathcal{S}} \cap \mathcal{S}$ is a vertex of $\mathcal{S}$ or $\conv(\ell_{\scriptscriptstyle\mathcal{S}} \cap \mathcal{S}) \subset \bb{R}^2$ is an edge of $\mathcal{S}$.

The next lemma is an immediate consequence of the existence of generating sets. Its proof is straightforward from definitions.

\begin{lemma}\label{lem_direxp}
If $\ell \in \bb{G}_1$ is an oriented line and there exists a set $\mathcal{S} \in$ $\mathcal{F}_{C}$ such that $\ell_{\scriptscriptstyle\mathcal{S}} \cap \mathcal{S} = \{g_0\}$ is $\eta$-generated by $\mathcal{S}$, then $\ell$ is a one-sided expansive dir\-ection on $X_{\eta}$.
\end{lemma}

\begin{remark}\label{Obs_irracional_line}
Given $\eta \in A^{\bb{Z}^2}$ with $P_{\eta}(\mathcal{U}) \leq |\mathcal{U}|+|A|-2$ for some $\mathcal{U} \in \mathcal{F}_{C}$, if $\ell \in \bb{G}_1$ is an irrational oriented line, then there exists an $\eta$-generating set $\mathcal{S} \in \mathcal{F}_{C}$ such that $\ell_{\scriptscriptstyle\mathcal{S}} \cap \mathcal{S} = \{g_0\}$ is $\eta$-generated by $\mathcal{S}$ and so, from Lemma~\ref{lem_direxp}, it follows that $\ell$ is a one-sided expansive direction on $X_{\eta}$. Furthermore, if $\ell \in \bb{G}_1$ is a one-sided nonexpansive direction on $X_{\eta}$, the above lemma also ensures that every $\eta$-generating set $\mathcal{S} \in \mathcal{F}^{V\!ol}_{C}$ has an edge parallel to $\ell$, i.e., $|\ell_{\scriptscriptstyle\mathcal{S}} \cap \mathcal{S}| \geq 2$.
\end{remark}

The next lemma is an immediate consequence of the compactness of the subshift.

\begin{lemma}\label{lem_dir_exp}
Let  $\ell \in \bb{G}_1$ be a rational oriented line. If $\ell$ is a one-sided expansive direction on $X_{\eta}$, then there is a set $\mathcal{S} \in \mathcal{F}_{C}^{V\!ol}$ such that $\ell_{\scriptscriptstyle\mathcal{S}} \cap \mathcal{S} = \{g_0\}$ is $\eta$-generated by $\mathcal{S}$.
\end{lemma}

For $\ell \in \bb{G}_1$, we call \emph{$\ell$-strip} any lattice translation of $\ell^t = \{g \in \bb{Z}^2 : \dist(g,\ell) \leq t\}$, where $t > 0$.

For $\eta \in A^{\bb{Z}^d}$ and $\mathcal{U} \subset \bb{Z}^d$ nonempty, we say that $\eta|\sob{\mathcal{U}} \in A^{\mathcal{U}}$ is \emph{periodic} of period $h \in (\bb{Z}^d)^*$ if $\eta_{g+h} = \eta_{g}$ for every $g \in \mathcal{U} \cap (\mathcal{U}-h)$. Clearly, this notion of periodicity extend the classical one. Given $\ell \in \bb{G}_1$, to indicate that a period $h$ belongs to $\ell \cap \bb{Z}^2$, we say that $\eta|\sob{\mathcal{U}}$ is \emph{$\ell$-periodic}.

For $x \in X_{\eta}$, if $\mathcal{S}$ is a subset of $\bb{R}^2$, we make a slight abuse of the notation by denoting $x|\sob{\mathcal{S}}$ instead of $x|\sob{\mathcal{S} \cap \bb{Z}^2}$.

\begin{proposition}\label{pps_par_antipar_exp}
If $\eta \in A^{\bb{Z}^2}$ is periodic, then $\ell \in \bb{G}_1$ is a nonexpansive line on $X_{\eta}$ if, and only if, $\ell$ and $\lle$ are both one-sided nonexpansive directions on $X_{\eta}$.
\end{proposition}
\begin{proof}
If $\ell \in \bb{G}_1$ is a nonexpansive line on $X_{\eta}$, then $\ell$ or $\lle$ is a one-sided nonexpansive direction on $X_{\eta}$. Suppose, by contradiction, that $\lle$ is a one-sided expansive direction on $X_{\eta}$. Thanks Lemma \ref{lem_dir_exp} there is a set $\mathcal{S} \in \mathcal{F}_{C}^{V\!ol}$ such that $\lle_{\scriptscriptstyle\mathcal{S}} \cap \mathcal{S} = \{g_0\}$ is $\eta$-ge\-nerated by $\mathcal{S}$. Let $F \subset \bb{Z}^2$ be an $\ell$-strip that contains a translation of $\mathcal{S}$ and consider a convex finite set $B \subset F$ such that, for any $g,g' \in \bb{Z}^2$, $(T^{g}\eta)|\sob{B} = (T^{g'}\eta)|\sob{B}$ implies $(T^{g}\eta)|\sob{F} = (T^{g'}\eta)|\sob{F}$. Such a subset exists because according to Lemma \ref{pps_perio_expan} the line $\ell$ contains all periods of $\eta$. Let $\hat{\ell} \in \bb{G}_1$ be the orthogonal line to $\ell$ and consider a non-null vector $v \in \hat{\ell} \cap \mathcal{H}(\ell)$. For any $\tau \in \bb{Z}$ there exist $\tau \leq t < t'  \leq \tau+P_{\eta}(B)$ with $(T^{tv}\eta)|\sob{B} = (T^{t'\!v}\eta)|\sob{B}$ and hence with $(T^{tv}\eta)|\sob{F} = (T^{t'\!v}\eta)|\sob{F}$. Since $g_0 \in \mathcal{S} \subset F$ is $\eta$-generated by $\mathcal{S}$, it is easy to argue by induction that the restriction of $\eta$ to the half plane $\mathcal{H}(\ell)$ is periodic of period $(t'-t)v$, where $t'-t \leq P_{\eta}(B)$. As $\tau$ is arbitrary, we conclude that $\eta$ is $\hat{\ell}$-periodic, but this contradicts Lemma \ref{pps_perio_expan}.
\end{proof}

We introduce a notion motivated by the existence of sets that are not necessary $\eta$-generating, but, with respect to a fixed direction, part of their vertices are $\eta$-gene- rated.

\begin{definition}\label{def_eta_gerador}
Given an oriented line $\ell \subset \bb{R}^2$, a set $\mathcal{U} \in \mathcal{F}_{C}$ is said to be an $(\eta,\ell)$-ge\-nerating set if each point of $\ell_{\scriptscriptstyle\mathcal{U}} \cap V(\mathcal{U})$ is $\eta$-generated by $\mathcal{U}$.
\end{definition}

Of course, every $\eta$-generating set $\mathcal{U} \in \mathcal{F}_{C}$ is, in particular, an $(\eta,\ell)$-generating set for every oriented line $\ell \subset \bb{R}^2$.

\begin{remark}\label{rem_exis_conj}
Given $\eta \in A^{\bb{Z}^2}$, suppose $P_{\eta}(\mathcal{U}) \leq |\mathcal{U}|+|A|-2$ for some set $\mathcal{U} \in \mathcal{F}_{C}$ and let $\ell \in \bb{G}_1$ be a rational oriented line. Set $\mathcal{S}_1 := \mathcal{U}$ and define $\mathcal{S}_{i+1} := \mathcal{S}_{i} \backslash \ell_{\scriptscriptstyle\mathcal{S}_i}$ for all $i \geq 1$. Let $I \geq 1$ be the greatest integer such that $P_{\eta}(\mathcal{S}_I) \leq |\mathcal{S}_I|+|A|-2$. Note that any convex set $\mathcal{S} \subset \mathcal{S}_I$ that is minimal among all convex sets $\mathcal{T} \subset \mathcal{S}_I$  that satisfy $\mathcal{S}_{I} \backslash \ell_{\scriptscriptstyle\mathcal{S}_I} \subset \mathcal{T}$ and $P_{\eta}(\mathcal{T}) \leq$ $|\mathcal{T}|+|A|-2$ is $(\eta,\ell)$-generating. Again, $1+|A|-2 < |A|$ implies that $\mathcal{S}$ has at least two points. Moreover, if $\mathcal{S} \backslash \ell_{\scriptscriptstyle\mathcal{S}}$ is not empty, then
\begin{enumerate}
	\item[(i)] $P_{\eta}(\mathcal{S}) - P_{\eta}(\mathcal{S} \backslash \ell_{\scriptscriptstyle\mathcal{S}}) \leq |\ell_{\scriptscriptstyle\mathcal{S}} \cap \mathcal{S}|-1$,
	\item[(ii)] there is a half plane $\mathcal{H} \subset \bb{Z}^2$ (whose edge is parallel to $\ell$) such that $\mathcal{S} \backslash \ell_{\scriptscriptstyle\mathcal{S}} =$ $\mathcal{U} \cap \mathcal{H}$.
\end{enumerate}
\end{remark}

Note that, if we also suppose in the previous remark that $\ell \in \bb{G}_1$ is a one-sided nonexpansive direction on $X_{\eta}$ and $\mathcal{S} \in \mathcal{F}^{V\!ol}_{C}$, then Remark \ref{Obs_irracional_line} provides the addi- tional property $|\ell_{\scriptscriptstyle\mathcal{S}} \cap \mathcal{S}| \geq 2$.

\section{Periodicity and balanced sets}
\label{sec4}

Balanced sets are in particular generating sets for which additional hypotheses are imposed on their geometry and on the bounds of their complexity. Such properties fit well in the context of the Alphabetical Morse-Hedlund Theorem. 


For a set $\mathcal{U} \in \mathcal{F}^{V\!ol}_{C}$, an oriented line $\ell \in \bb{G}_1$ and $\gamma \in L(\mathcal{U} \backslash \ell_{\scriptscriptstyle\mathcal{U}},\eta)$, we define
\begin{equation}\label{eq_dfn_cont}
N_{\mathcal{U}}(\ell,\gamma) := \left|\{\gamma' \in L(\mathcal{U},\eta) : \gamma'|\sob{\mathcal{U} \backslash \ell_{\scriptscriptstyle\mathcal{U}}} = \gamma\}\right|.
\end{equation} 
In particular, $N_{\mathcal{U}}(\ell,\gamma) = 1$ means that 
$\gamma'|\sob{\mathcal{U} \backslash \ell_{\scriptscriptstyle\mathcal{U}}} = \gamma = \gamma''|\sob{\mathcal{U} \backslash \ell_{\scriptscriptstyle\mathcal{U}}}$ implies $\gamma' = \gamma''$ for any $\mathcal{U}$-configurations $\gamma',\gamma'' \in L(\mathcal{U},\eta)$. If $\ell \in \bb{G}_1$ is a rational oriented line, for a configuration $x \in X_{\eta}$, let $$L^{\ell}(\mathcal{U},x) := \left\{(T^{t\vec{v}_{\ell}}x)|\sob{\mathcal{U}} \ : t \in \bb{Z}\right\}.$$

\begin{lemma}\label{lem_exist_x}
For $\eta \in A^{\bb{Z}^2}$\! and a rational oriented line $\ell \in \bb{G}_1$, let $\mathcal{U} \in$ $\mathcal{F}^{V\!ol}_{C}$ be an $(\eta,\ell)$-generating set. If $x \in X_{\eta}$ and there is $\gamma \in L^{\ell}(\mathcal{U} \backslash \ell_{\scriptscriptstyle\mathcal{U}},x)$ such that $N_{\mathcal{U}}(\ell,\gamma) = 1$, then, for every $\ell$-strip $F \supset \mathcal{U} \backslash \ell_{\scriptscriptstyle\mathcal{U}}$ and any configuration $y \in X_{\eta}$, $x|\sob{F} =$ $y|\sob{F}$ implies that $x|\sob{\ell_{\scriptscriptstyle\mathcal{U}} \cup F} = y|\sob{\ell_{\scriptscriptstyle\mathcal{U}} \cup F}$.
\end{lemma}
\begin{proof}
Since $\gamma \in L^{\ell}(\mathcal{U} \backslash \ell_{\scriptscriptstyle\mathcal{U}},x)$, there exists $\tau \in \bb{Z}$ such that $\gamma = (T^{\tau\vec{v}_{\ell}}x)|\sob{\mathcal{U} \backslash \ell_{\scriptscriptstyle\mathcal{U}}}$. So, from $x|\sob{F} = y|\sob{F}$, we obtain $(T^{\tau\vec{v}_{\ell}}x)|\sob{\mathcal{U} \backslash \ell_{\scriptscriptstyle\mathcal{U}}} = \gamma = (T^{\tau\vec{v}_{\ell}}y)|\sob{\mathcal{U} \backslash \ell_{\scriptscriptstyle\mathcal{U}}}$. As $N_{\mathcal{U}}(\ell,\gamma) = 1$, we have then $(T^{\tau\vec{v}_{\ell}}x)|\sob{\mathcal{U}} = (T^{\tau\vec{v}_{\ell}}y)|\sob{\mathcal{U}}$, which implies
\begin{equation}\label{eq_pps_exist_x}
x|\sob{(\mathcal{U}+\tau \vec{v}_{\ell}) \cup F} = y|\sob{(\mathcal{U}+\tau \vec{v}_{\ell}) \cup F}.
\end{equation}
By hypothesis, each vertices of $\mathcal{U}$ in $\ell_{\scriptscriptstyle\mathcal{U}} \cap \mathcal{U}$ are $\eta$-generated by $\mathcal{U}$. Hence, from (\ref{eq_pps_exist_x}) it follows by induction that $x|\sob{\ell_{\scriptscriptstyle\mathcal{U}} \cup F} = y|\sob{\ell_{\scriptscriptstyle\mathcal{U}} \cup F}$, which completes the proof.
\end{proof}

Clearly, if $\ell \in \bb{G}_1$ is a one-sided nonexpansive direction on $X_{\eta}$ and $\mathcal{U} \in \mathcal{F}^{V\!ol}_{C}$ is an $(\eta,\ell)$-generating set, then there are configurations $x \in X_{\eta}$ such that
\begin{equation}\label{eq_pontos_naoexpan}
N_{\mathcal{U}}(\ell,\gamma) > 1 \quad \forall \  \mathcal{U} \backslash \ell_{\scriptscriptstyle\mathcal{U}}\textrm{-configuration} \ \gamma \in L^{\ell}(\mathcal{U} \backslash \ell_{\scriptscriptstyle\mathcal{U}},x).
\end{equation}
We denote by $\mathcal{M}(\ell,\mathcal{U})$ the set formed by the configurations $x \in X_{\eta}$ that satisfy (\ref{eq_pontos_naoexpan}). The configurations that belong to $\mathcal{M}(\ell,\mathcal{U})$ are exactly the ones for which each $\mathcal{U} \backslash \ell_{\scriptscriptstyle\mathcal{U}}$-restriction admits multiple extensions to $\mathcal{U}$. It is not difficult to see that $P_{\eta}(\mathcal{U}) - P_{\eta}(\mathcal{U} \backslash \ell_{\scriptscriptstyle\mathcal{U}}) = \sum_{\gamma \in L(\mathcal{U} \backslash \ell_{\scriptscriptstyle\mathcal{U}},\eta)} (N_{\mathcal{U}}(\ell,\gamma)-1)$,
which, for each $x \in \mathcal{M}(\ell,\mathcal{U})$, yields  
\begin{equation}\label{desig_card}
P_{\eta}(\mathcal{U})-P_{\eta}(\mathcal{U} \backslash \ell_{\scriptscriptstyle\mathcal{U}}) \geq \sum_{\gamma \in L^{\ell}(\mathcal{U} \backslash \ell_{\scriptscriptstyle\mathcal{U}},x)} \big(N_{\mathcal{U}}(\ell,\gamma)-1\big) \geq \big|L^{\ell}(\mathcal{U} \backslash \ell_{\scriptscriptstyle\mathcal{U}},x)\big|.
\end{equation}

By applying Morse-Hedlund Theorem, Cyr and Kra showed that, under certain conditions, a suitable upper bound for $P_{\eta}(\mathcal{U})-P_{\eta}(\mathcal{U} \backslash \ell_{\scriptscriptstyle\mathcal{U}})$ and so for $|L^{\ell}(\mathcal{U} \backslash \ell_{\scriptscriptstyle\mathcal{U}},x)|$, imposes periodicity in some regions of specific configurations. In the sequel, we will employ a similar strategy.

Let $\ell \in \bb{G}_1$ be an oriented line. Given a set $\mathcal{U} \in \mathcal{F}^{V\!ol}_{C}$, for each oriented line $\ell' \subset \bb{R}^2$ parallel to $\ell$ such that $\ell' \cap \mathcal{U} \neq \emptyset$, we denote $i_{\mathcal{U}}(\ell')$ and $f_{\mathcal{U}}(\ell')$, respectively, the initial and the final points on $\ell' \cap \mathcal{U}$ according to the orientation of $\ell'$. If $|\ell' \cap \mathcal{U}|$ is equal to $1$, then $i_{\mathcal{U}}(\ell')$ and $f_{\mathcal{U}}(\ell')$ are the same point. For $p \in \bb{N}$, we define $$\mathcal{I}^{\ell,p}(\mathcal{U}) := \left\{i_{\mathcal{U}}(\ell') \in \bb{Z}^2 : \ell' \ \textrm{is parallel to} \ \ell, \ \ell' \neq \ell_{\scriptscriptstyle\mathcal{U}}, \ |\ell' \cap \mathcal{U}| \geq p\right\}$$ and $$\mathcal{F}^{\ell,p}(\mathcal{U}) := \left\{f_{\mathcal{U}}(\ell') \in \bb{Z}^2 : \ell' \ \textrm{is parallel to} \ \ell, \ \ell' \neq \ell_{\scriptscriptstyle\mathcal{U}}, \ |\ell' \cap \mathcal{U}| \geq p\right\}.$$ We call \emph{$(\ell,\mathcal{U},p)$-strip} the set $\bigcup_{t \in \bb{Z}} (\mathcal{I}^{\ell,p}(\mathcal{U})+t\vec{v}_{\ell})$. For $x \in X_{\eta}$, we denote by $$A^{\ell,p}(\mathcal{U},x) := \Big\{(T^{t\vec{v}_{\ell}}x)|\sob{\mathcal{I}^{\ell,p}(\mathcal{U})} : t \in \bb{Z}\Big\}$$ the finite alphabet induced by the sequence $\xi_{t} := (T^{t\vec{v}_{\ell}}x)|\sob{\mathcal{I}^{\ell,p}(\mathcal{U})}$, where $t \in \bb{Z}$. This sequence is closely related to the restriction of the configuration $x$ to the $(\ell,\mathcal{U},p)$-strip and this viewpoint will be essential in several arguments.

Similarly to Lemma 2.24 of \cite{van}, the following result shows how nonexpansiveness is connected to periodicity in certain regions of some configurations.

\begin{lemma}\label{lem_exist_faixa_period}
Given $\eta \in A^{\bb{Z}^2}$, suppose $\ell \in \bb{G}_1$ is a one-sided nonexpansive direction on $X_{\eta}$ and $\mathcal{U} \in$  $\mathcal{F}^{V\!ol}_{C}$ is an $(\eta,\ell)$-generating set. If $x \in \mathcal{M}(\ell,\mathcal{U})$ and there is $p \in \bb{N}$ such that $P_{\eta}(\mathcal{U}) - P_{\eta}(\mathcal{U} \backslash \ell_{\scriptscriptstyle\mathcal{U}}) \leq$ $p+|A^{\ell,p}(\mathcal{U},x)|-2$, then the restriction of $x$ to the $(\ell,\mathcal{U},p)$-strip is periodic of period $t\vec{v}_{\ell}$ for some $t \leq$ $p+|A^{\ell,p}(\mathcal{U},x)|-2$.
\end{lemma}
\begin{proof}
Note that, from (\ref{desig_card}) and by hypothesis, it follows that  $1 \leq |L^{\ell}(\mathcal{U} \backslash \ell_{\scriptscriptstyle\mathcal{U}},x)| \leq p+|A^{\ell,p}(\mathcal{U},x)|-2$.
 Set $\mathcal{R} := \bigcup_{t=0}^{p-1} (\mathcal{I}^{\ell,p}(\mathcal{U})+t\vec{v}_{\ell}) \subset \mathcal{U}$ and let $\xi = (\xi_t)_{t \in \bb{Z}}$ be the sequence defined by $\xi_{t} = (T^{t\vec{v}_{\ell}}x)|\sob{\mathcal{I}^{\ell,p}(\mathcal{U})}$ for all $t \in \bb{Z}$. Note that, for every $\tau \in \bb{Z}$, the word $\xi_{\tau}\xi_{{\tau}+1} \cdots \xi_{{\tau}+p-1}$ is naturally identified with the $\mathcal{R}$-configuration $(T^{{\tau}\vec{v}_{\ell}}x)|\sob{\mathcal{R}} \in$ $L^{\ell}(\mathcal{R},x)$. If $|A^{\ell,p}(\mathcal{U},x)| = 1$, then there is nothing to argue. Otherwise, since $$P_{\xi}(p) = |L^{\ell}(\mathcal{R},x)| \leq |L^{\ell}(\mathcal{U} \backslash \ell_{\scriptscriptstyle\mathcal{U}},x)| \leq p+|A^{\ell,p}(\mathcal{U},x)|-2,$$ the Alphabetical Morse-Hedlund Theorem ensures that the sequence $\xi$ is periodic of period at most $p+|A^{\ell,p}(\mathcal{U},x)|-2$. Thus, the restriction of $x$ to the $(\ell,\mathcal{U},p)$-strip is periodic of period $t\vec{v}_{\ell}$ for some $t \leq p+|A^{\ell,p}(\mathcal{U},x)|-2$.
\end{proof}

We will need a version of the above lemma for half-strips. So let $\ell \in \bb{G}_1$ be a rational oriented line. If $\mathcal{U} \in \mathcal{F}^{V\!ol}_{C}$ and $a \in \bb{Z}$, for each $x \in X_{\eta}$, we define $$L^{\ell}_{a+}(\mathcal{U},x) := \left\{(T^{t\vec{v}_{\ell}}x)|\sob{\mathcal{U}} \ : t \geq a\right\} \quad \textrm{and} \quad L^{\ell}_{a-}(\mathcal{U},x) := \left\{(T^{-t\vec{v}_{\ell}}x)|\sob{\mathcal{U}} \ : t \geq a\right\}.$$ Recall that $-\vec{v}_{\ell}$ is parallel to $\lle \in \bb{G}_1$. Naturally, we are led to consider con\-fi\-gu\-rations $x \in X_{\eta}$ such that
\begin{equation}\label{eq_pontos_naoexpan_sfaixa_v1}
N_{\mathcal{U}}(\ell,\gamma) > 1 \quad \forall \ \mathcal{U} \backslash \ell_{\scriptscriptstyle\mathcal{U}}\textrm{-configuration} \ \gamma \in L^{\ell}_{a+}(\mathcal{U} \backslash \ell_{\scriptscriptstyle\mathcal{U}},x) 
\end{equation}
or
\begin{equation}\label{eq_pontos_naoexpan_sfaixa_v2}
N_{\mathcal{U}}(\ell,\gamma) > 1 \quad \forall \ \mathcal{U} \backslash \ell_{\scriptscriptstyle\mathcal{U}}\textrm{-configuration} \ \gamma \in L^{\ell}_{a-}(\mathcal{U} \backslash \ell_{\scriptscriptstyle\mathcal{U}},x). 
\end{equation}
We denote by $\mathcal{M}_{a+}(\ell,\mathcal{U})$ and $\mathcal{M}_{a-}(\ell,\mathcal{U})$ the sets formed by the configurations $x \in X_{\eta}$ that satisfy,  respectively, (\ref{eq_pontos_naoexpan_sfaixa_v1}) and (\ref{eq_pontos_naoexpan_sfaixa_v2}). Note that $\mathcal{M}(\ell,\mathcal{U}) \subset \mathcal{M}_{a+}(\ell,\mathcal{U}) \cap \mathcal{M}_{a-}(\ell,\mathcal{U})$. It is clear that analogous inequalities as (\ref{desig_card}) also hold in this context, i.e., for $x \in \mathcal{M}_{a\pm}(\ell,\mathcal{U})$, one has
\begin{equation}\label{eq_card_a}
P_{\eta}(\mathcal{U})-P_{\eta}(\mathcal{U} \backslash \ell_{\scriptscriptstyle\mathcal{U}}) \geq \big|L^{\ell}_{a\pm}(\mathcal{U} \backslash \ell_{\scriptscriptstyle\mathcal{U}},x)\big|.
\end{equation}
Given $p \in \bb{N}$, for each integer $a \in \bb{Z}$, we call \emph{$(\ell,\mathcal{U},p)$-half-strips} the sets defined by
\begin{equation}\label{def_semi_faixa}
F^{+}(a) := \bigcup_{t \geq a} (\mathcal{I}^{\ell,p}(\mathcal{U})+t\vec{v}_{\ell}) \ \ \textrm{and}  \ \ F^{-}(a) := \bigcup_{t \geq a} (\mathcal{F}^{\ell,p}(\mathcal{U})-t\vec{v}_{\ell}).
\end{equation} 
For $x \in X_{\eta}$, let $$A^{\ell,p}_{a+}(\mathcal{U},x) := \Big\{(T^{t\vec{v}_{\ell}}x)|\sob{\mathcal{I}^{\ell,p}(\mathcal{U})} : t \geq a\Big\}$$ and $$A_{a-}^{\ell,p}(\mathcal{U},x) := \Big\{(T^{-t\vec{v}_{\ell}}x)|\sob{\mathcal{F}^{\ell,p}(\mathcal{U})} : t \geq a\Big\}$$ denote the induced finite alphabets.

\begin{lemma}\label{lem_exist_sfaixa_period}
Given $\eta \in A^{\bb{Z}^2}$, suppose $\ell \in \bb{G}_1$ is a one-sided nonexpansive direction on $X_{\eta}$ and that $\mathcal{U} \in \mathcal{F}^{V\!ol}_{C}$ is an $(\eta,\ell)$-generator set. If $x \in \mathcal{M}_{a\pm}(\ell,\mathcal{U})$ and there exists $p \in \bb{N}$ such that $P_{\eta}(\mathcal{U}) - P_{\eta}(\mathcal{U} \backslash \ell_{\scriptscriptstyle\mathcal{U}}) \leq p+|A^{\ell,p}_{a\pm}(\mathcal{U},x)|-2$, then the restriction of $x$ to the $(\ell,\mathcal{U},p)$-half-strip $F^{\pm}(a+m_x)$ is periodic of pe\-ri\-od $\pm t\vec{v}_{\ell}$ for some $t \leq m_x :=$ $p+|A^{\ell,p}_{a\pm}(\mathcal{U},x)|-2$.
\end{lemma}
\begin{proof}
The proof is identical to that of the previous lemma, by taking (\ref{eq_card_a}) instead of (\ref{desig_card}) and using condition (ii) instead of condition (iii) in the Alphabetical Morse-Hedlund Theorem.
\end{proof}

If the set $\mathcal{U} \in \mathcal{F}^{V\!ol}_{C}$ does not have suitable geometrical properties, an $(\ell,\mathcal{U},p)$-(half-)strip might be a ``nonconnected'' set, in the sense that its convex hull contains further points of $\mathcal{U}$ than the (half-)strip itself. In such a situation, it could, for example, not be possible to extend to a half-plane the periodicity obtained for an $(\ell,\mathcal{U},p)$-strip in Lemma \ref{lem_exist_faixa_period}. We will prevent this inconvenience by imposing that lines parallel to $\ell_{\scriptscriptstyle\mathcal{U}}$ intersect $\mathcal{U}$ in a sufficient number of points (see condition (i) in the next definition). This leads us to consider balanced sets, a class of generating sets also obeying the bounds on their complexity highlighted in the previous lemmas (see condition (ii) below). 

\begin{definition}\label{dfn_lp_balanceado}
Given $\eta \in A^{\bb{Z}^2}$, let $\ell \in \bb{G}_1$ be a one-sided nonexpansive direction on $X_{\eta}$. An $(\eta,\ell)$-generating set $\mathcal{U} \in \mathcal{F}^{V\!ol}_{C}$ is said to be $(\ell,p)$-balanced (on $X_{\eta}$), $p \in \bb{N}$, if the following conditions hold:
\begin{enumerate}
	\item[(i)] for every $\ell' \neq \ell_{\scriptscriptstyle\mathcal{U}}$ that contains at least one point of $\bb{Z}^2$ and is parallel to $\ell$, $|\ell' \cap \mathcal{U}| \geq p$ whenever $\ell' \cap \conv(\mathcal{U})$ $\neq \emptyset$,
	\item[(ii)] for each $x \in \mathcal{M}(\ell,\mathcal{U})$ with $|A^{\ell,p}(\mathcal{U},x)| > 1$, we have that $P_{\eta}(\mathcal{U}) - P_{\eta}(\mathcal{U} \backslash \ell_{\scriptscriptstyle\mathcal{U}}) \leq p_x+|A^{\ell,p_x}(\mathcal{U},x)|-2$ for some positive integer $p_x \leq p$.
\end{enumerate}
\end{definition}

Under an appropriate bound on the complexity, Proposition \ref{pps_lp-balanceado} ensures the exis\-tence of balanced sets. Furthermore, as $x \in \mathcal{M}(\ell,\mathcal{U}+u)$ if, and only if, $T^{u}x \in \mathcal{M}(\ell,\mathcal{U})$, it is easy to argue that the property of being an $(\ell,p)$-balanced set is invariant by translations. If $\mathcal{U} \in \mathcal{F}^{V\!ol}_{C}$ satisfies condition (i) of Definition \ref{dfn_lp_balanceado}, it is im\-me\-di\-ate that $A^{\ell,p}(\mathcal{U},x) =$ $A^{\ell,p_x}(\mathcal{U},x)$ for every positive integer $p_x \leq p$ and $x \in X_{\eta}$.

\begin{remark}\label{rem_balan_semifaixa1}
Let $\mathcal{U} \in \mathcal{F}^{V\!ol}_{C}$ be an $(\ell,p)$-balanced set. If $x \in \mathcal{M}_{a\pm}(\ell,\mathcal{U})$ satisfies $|A_{s\pm}^{\ell,p}(\mathcal{U},x)| > 1$ for infinitely many integer $s \geq a$, then it is easy to see that there exists a positive integer $p_{x} \leq p$ such that $P_{\eta}(\mathcal{U}) - P_{\eta}(\mathcal{U} \backslash \ell_{\scriptscriptstyle\mathcal{U}}) \leq p_x+|A_{a\pm}^{\ell,p_x}(\mathcal{U},x)|-2$.
\end{remark}

If $\varpi,\varpi' \in E(\mathcal{U})$ are antiparallel edges, where $\mathcal{U} \in \mathcal{F}^{V\!ol}_{C}$ is a quasi-regular set (see Definition \ref{q-regular}),  let $R,S \subset \bb{R}^2$ denote the line segments connecting the initial and the final points of $\varpi$ and $\varpi'$. Line segments such as $R,S \subset \bb{R}^2$ are called \emph{axis of symmetry of $\mathcal{U}$}. It is easy to argue that each axis of symmetry $S \subset \bb{R}^2$ divides $\mathcal{U}$ in two subsets $A_S$ and $B_S $ with $A_S \cap B_S = S \cap \bb{Z}^2$, $A_S \cup B_S = \mathcal{U}$ and $|A_S| = |B_S|$.

\begin{remark}\label{rem_cap2_arest_paral}
Let $\mathcal{U} \in \mathcal{F}^{V\!ol}_{C}$ and suppose $\varpi,\varpi' \in E(\mathcal{U})$ are antiparallel edges. For any oriented line $\ell \subset \bb{R}^2$ parallel to $\varpi$ that intersects $\conv(\mathcal{U})$ and contains at least one point of $\bb{Z}^2$, if $|\varpi \cap \mathcal{U}|$ $\leq |\varpi' \cap \mathcal{U}|$, since the length of $\varpi$ is less or equal to the length of the line segment $\ell \cap \conv(\mathcal{U})$, then $|\ell \cap \mathcal{U}| \geq |\varpi \cap \mathcal{U}|-1$.
\end{remark}

The following result shows how a strong complexity assumption ensures the exis- tence of balanced sets for any one-sided nonexpansive direction.

\begin{proposition}
\label{pps_lp-balanceado}Given $\eta \in A^{\bb{Z}^2}$ aperiodic, suppose there exists a quasi-regular set $\mathcal{U} \in \mathcal{F}^{V\!ol}_{C}$ for which $P_{\eta}(\mathcal{U}) \leq \frac{1}{2}|\mathcal{U}|+|A|-1$. If $\ell \in \bb{G}_1$ is a one-sided nonexpansive direction on $X_{\eta}$, then there exists an $(\eta,\ell)$-generating set $\mathcal{S} \in \mathcal{F}^{V\!ol}_{C}$ such that
\begin{enumerate}
	\item[(i)] $|\ell_{\scriptscriptstyle\mathcal{S}} \cap \mathcal{S}| \leq |\lle_{\scriptscriptstyle\mathcal{S}} \cap \mathcal{S}|$,
	\item[(ii)] $P_{\eta}(\mathcal{S}) - P_{\eta}(\mathcal{S} \backslash \ell_{\scriptscriptstyle\mathcal{S}}) \leq |\ell_{\scriptscriptstyle\mathcal{S}} \cap \mathcal{S}|-1$. 	
\end{enumerate}
In particular, it follows that $\mathcal{S} \in \mathcal{F}^{V\!ol}_{C}$ is an $(\ell,p)$-balanced set with $p = |\ell_{\scriptscriptstyle\mathcal{S}} \cap \mathcal{S}|-1$.
\end{proposition}
\begin{proof}
Initially, let $z \in \bb{R}^2$ be the intersection of two distinct axis of symmetry of $\mathcal{U}$. Since $\mathcal{U}$ is a quasi-regular set, the oriented line parallel to $\ell$ that passes through the point $z \in \bb{R}^2$ intersects antiparallel edges $\varpi,\varpi' \in E(\mathcal{U})$. Let $\lle' \subset \bb{R}^2$ be the oriented line antiparallel to $\ell$ for which the half plane $\mathcal{H}(\lle')$ is maximal among all half planes whose edge (antiparallel to $\ell$) has nonempty intersection with both $\varpi$ and $\varpi'$. By maximality, there exists a vertex on $\varpi$ or $\varpi'$, denoted by $u \in V(\mathcal{U})$, such that $u \in \ell' \, \cap$ $\bb{Z}^2$. In particular, at most one of the initial and final points of $\varpi$ and $\varpi'$ may not belong to $\mathcal{H}(\lle')$. So if $S \subset \bb{R}^2$ is the axis of symmetry with $u \in S$, then $\mathcal{H}(\lle')$ contains one of the sets $A_S,B_S \subset \mathcal{U}$. Without loss of generality, we assume that $B_S \subset \mathcal{H}(\lle')$ (see Figure \ref{exist_lp_balan}). Then, since $|A_S| = |B_S|$, we conclude that $$|\mathcal{U} \cap \mathcal{H}(\lle')| \geq |B_S| \geq |B_S|-\frac{1}{2}|A_S \cap B_S|+1 = \frac{1}{2}|\mathcal{U}|+1.$$ Denoting $\mathcal{T} := \mathcal{U} \cap \mathcal{H}(\lle')$, we have then $$P_{\eta}(\mathcal{T})-|\mathcal{T}| \leq P_{\eta}(\mathcal{U})-|\mathcal{T}| \leq \frac{1}{2}|\mathcal{U}|+|A|-1-|\mathcal{T}| \leq \frac{1}{2}|\mathcal{U}|+|A|-1-\frac{1}{2}|\mathcal{U}|-1,$$ so that $P_{\eta}(\mathcal{T}) \leq |\mathcal{T}|+|A|-2$.
\begin{figure}[ht]
	\centering
	\def\svgwidth{2.8cm}
	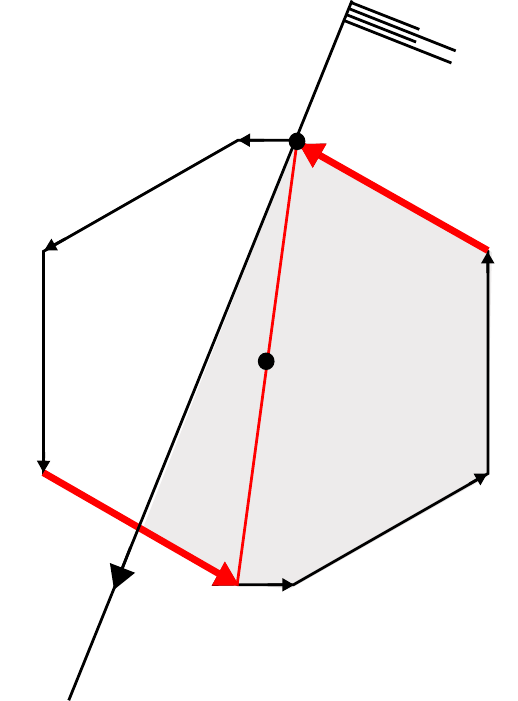
	\caption{The line \hspace{1.1cm} and the subsets $A_S,B_S \subset \mathcal{U}$.}
	\label{exist_lp_balan}
	\vspace{-13pt}\hspace{-1.3cm}$\lle' \subset \bb{R}^2$
\end{figure}

Since (by construction) $\ell'$ has nonempty intersection with $\varpi,\varpi' \in E(\mathcal{U})$ and $\mathcal{U}$ is a con\-vex set, the length of $\ell'' \cap \conv(\mathcal{U})$ is less or equal to the length of $\ell' \cap \conv(\mathcal{U})$ for all oriented line $\ell'' \subset \bb{R}^2$ parallel to $\ell$. This means that 
\begin{equation}\label{eq_max_cardinalidade}
|\ell' \cap \mathcal{U}| = \max\left\{|\ell'' \cap \mathcal{U}| : \ell'' \ \textrm{is parallel to} \ \ell\right\}.
\end{equation}

According to Remark \ref{rem_exis_conj}, there exists a half plane $\mathcal{H} \subset \bb{Z}^2$ (whose edge is paral- lel to $\ell$) and an $(\eta,\ell)$-generating set $\mathcal{S} \in \mathcal{F}^{V\!ol}_{C}$ such that
\begin{equation}\label{eq_des_ret_antip}
\mathcal{S} \backslash \ell_{\scriptscriptstyle\mathcal{S}} =  \left(\mathcal{U} \cap \mathcal{H}(\lle')\right) \cap \mathcal{H} = \mathcal{T} \cap \mathcal{H},
\end{equation}
\begin{equation}\label{eq_des_principal_antip}
P_{\eta}(\mathcal{S}) - P_{\eta}(\mathcal{S} \backslash \ell_{\scriptscriptstyle\mathcal{S}}) \leq |\ell_{\scriptscriptstyle\mathcal{S}} \cap \mathcal{S}|-1.
\end{equation} 
Note that $\lle_{\mathcal{S}} = \lle_{\mathcal{S} \backslash \ell_{\scriptscriptstyle\mathcal{S}}} = \lle_{(\mathcal{U} \cap \mathcal{H}(\lle')) \cap \mathcal{H}} = \lle_{\mathcal{U} \cap \mathcal{H}(\lle')} = \lle'$, where the first equality holds for any set $\mathcal{S} \in \mathcal{F}^{V\!ol}_{C}$ and the third follows because the edge of $\mathcal{H}$ is antiparallel to $\lle'$. Furthermore, from (\ref{eq_des_ret_antip}) we get $\lle_{\scriptscriptstyle\mathcal{S}} \cap \mathcal{S} = \lle_{\scriptscriptstyle\mathcal{S}} \cap (\mathcal{S} \backslash \ell_{\scriptscriptstyle\mathcal{S}}) = \lle' \cap \mathcal{U}$. So by (\ref{eq_max_cardinalidade}) one has $$|\lle_{\scriptscriptstyle\mathcal{S}} \cap \mathcal{S}| = |\ell' \cap \mathcal{U}| \geq |\ell_{\scriptscriptstyle\mathcal{S}} \cap \mathcal{U}| \geq |\ell_{\scriptscriptstyle\mathcal{S}} \cap \mathcal{S}|.$$ 

Finally, we claim that $\mathcal{S}$ is an $(\ell,p)$-balanced set with $p := |\ell_{\scriptscriptstyle\mathcal{S}} \cap \mathcal{S}|-1$. Indeed, Remark \ref{rem_cap2_arest_paral} ensures condition (i) of Definition \ref{dfn_lp_balanceado}, and from inequality~\eqref{eq_des_principal_antip}, for each $x \in \mathcal{M}(\ell,\mathcal{S})$ with $|A^{\ell,p}(\mathcal{S},x)| > 1$,  we get $P_{\eta}(\mathcal{S}) - P_{\eta}(\mathcal{S} \backslash \ell_{\scriptscriptstyle\mathcal{S}}) \leq p \leq p + |A^{\ell,p}(\mathcal{S},x)|-2$, i.e., condition (ii) of Definition~\ref{dfn_lp_balanceado} holds with $p_x = p$.
\end{proof}

The notion of balanced sets introduced here has some advantages.  For example, if $\ell \in \bb{G}_1$ is a one-sided nonexpansive direction on $X_{\eta}$ and $\mathcal{S} \in \mathcal{F}^{V\!ol}_{C}$ is an $(\eta,\ell)$-ge- nerating set where $P_{\eta}(\mathcal{S}) - P_{\eta}(\mathcal{S} \backslash \ell_{\scriptscriptstyle\mathcal{S}}) \leq |\ell_{\scriptscriptstyle\mathcal{S}} \cap \mathcal{S}|-1$, then $\mathcal{S}$ may be $(\ell,p)$-balanced even when $|\ell_{\scriptscriptstyle\mathcal{S}} \cap \mathcal{S}| > |\lle_{\scriptscriptstyle\mathcal{S}} \cap \mathcal{S}|$. Indeed, if $|\lle_{\scriptscriptstyle\mathcal{S}} \cap \mathcal{S}| > 1$, it is enough to have $$|\ell_{\scriptscriptstyle\mathcal{S}} \cap \mathcal{S}| \leq |\lle_{\scriptscriptstyle\mathcal{S}} \cap \mathcal{S}|+|A^{\ell,p}(\mathcal{S},x)|-2$$ for every configuration $x \in \mathcal{M}(\ell,\mathcal{S})$ with $|A^{\ell,p}(\mathcal{S},x)| > 1$, where $p := |\lle_{\scriptscriptstyle\mathcal{S}} \cap \mathcal{S}|-1$.

\subsection{Extending periodicity from strips}

For an $(\ell,p)$-balanced set $\mathcal{T} \in \mathcal{F}^{V\!ol}_{C}$, we define $$\varPhi_{p}(\ell,\mathcal{T}) := P_{\eta}(\mathcal{T})-P_{\eta}(\mathcal{T} \backslash \ell_{\scriptscriptstyle\mathcal{T}})$$ if $|A^{\ell,p}(\mathcal{T},x)| = 1$ for every $x \in \mathcal{M}(\ell,\mathcal{T})$ and $$\varPhi_{p}(\ell,\mathcal{T}) := \max\left\{p_x+|A^{\ell,p_x}(\mathcal{T},x)|-2 : x \in \mathcal{M}(\ell,\mathcal{T}) \ \textrm{and} \  |A^{\ell,p_x}(\mathcal{T},x)| > 1\right\}$$ otherwise, where $p_x \leq p$ is the smallest integer fulfilling Definition~\ref{dfn_lp_balanceado}. Although the definition of $\varPhi_{p}(\ell,\mathcal{T})$ may seem somewhat elaborate, we prefer to use it to make it clear that we can extend periodicity beyond strips even without explicitly using any hypothesis about complexity. When we consider our alphabetical upper bound for complexity, however, the function $\varPhi_{p}(\ell,\mathcal{T})$ may be replaced by a nice expression (see the inequality (\ref{maineq_sec4})).

The next lemma will allow us to extend the periodicity to larger strips and then to half planes.

\begin{lemma}\label{lema_extens}
Given $\eta \in A^{\bb{Z}^2}$, suppose $\ell \in \bb{G}_1$ is a one-sided nonexpansive direction on $X_{\eta}$ and $\mathcal{U} \in \mathcal{F}^{V\!ol}_{C}$ is an $(\ell,p)$-balanced set. If the restriction of $x \in X_{\eta}$ to the $(\ell,\mathcal{U},p)$-strip $F$ is periodic of period $t'\vec{v}_{\ell}$ for some $t' \in \bb{N}$, then $x|\sob{\ell_{\scriptscriptstyle\mathcal{U}} \cup F}$ is periodic of period $t\vec{v}_{\ell}$, where $t = t'$ if $x \not \in \mathcal{M}(\ell,\mathcal{U})$ and $t \leq$ $2\varPhi_{p}(\ell,\mathcal{U})$ otherwise.
\end{lemma}
\begin{proof}
Initially, if $x \not\in \mathcal{M}(\ell,\mathcal{U})$, since $x|\sob{F}$ is periodic of period $t'\vec{v}_{\ell}$, that is, $x|\sob{F} =$ $(T^{t'\vec{v}_{\ell}}x)|\sob{F}$, from Lemma \ref{lem_exist_x} it follows that $x|\sob{\ell_{\scriptscriptstyle\mathcal{U}} \cup F} = (T^{t'\vec{v}_{\ell}}x)|\sob{\ell_{\scriptscriptstyle\mathcal{U}} \cup F}$.

Suppose $x \in \mathcal{M}(\ell,\mathcal{U})$ and $|A^{\ell,p}(\mathcal{U},x)| > 1$. In this case, from (\ref{desig_card}) we obtain that $|L^{\ell}(\mathcal{U} \backslash \ell_{\scriptscriptstyle\mathcal{U}},x)| \leq$ $p_x+\big|A^{\ell,p_x}(\mathcal{U},x)\big|-2 \leq \varPhi_{p}(\ell,\mathcal{U})$ where $p_x \leq p$ is the smallest integer fulfilling Definition~\ref{dfn_lp_balanceado}. Therefore,
\begin{align}\label{des_lemma_princ}
\big|L^{\ell}(\mathcal{U},x)\big|-\varPhi_{p}(\ell,\mathcal{U}) & \leq \left(\sum_{\gamma \in L^{\ell}(\mathcal{U} \backslash \ell_{\scriptscriptstyle\mathcal{U}},x)} N_{\mathcal{U}}(\ell,\gamma)\right)-\big|L^{\ell}(\mathcal{U} \backslash \ell_{\scriptscriptstyle\mathcal{U}},x)\big|\nonumber\\ & = \sum_{\gamma \in L^{\ell}(\mathcal{U} \backslash \ell_{\scriptscriptstyle\mathcal{U}},x)} \Big(N_{\mathcal{U}}(\ell,\gamma)-1\Big)\nonumber\\ & \leq \sum_{\gamma \in L(\mathcal{U} \backslash \ell_{\scriptscriptstyle\mathcal{U}},\eta)} \Big(N_{\mathcal{U}}(\ell,\gamma)-1\Big)\nonumber\\ & = P_{\eta}(\mathcal{U})-P_{\eta}(\mathcal{U} \backslash \ell_{\scriptscriptstyle\mathcal{U}}) \leq p_x+\big|A^{\ell,p_x}(\mathcal{U},x)\big|-2,
\end{align}
which yields
\begin{equation*}
|L^{\ell}(\mathcal{U},x)| \leq \varPhi_{p}(\ell,\mathcal{U})+p_x+|A^{\ell,p_x}(\mathcal{U},x)|-2 \leq 2\varPhi_{p}(\ell,\mathcal{U}).
\end{equation*}
By the Pigeonhole Principle, we can assume, without loss of generality, that there exists a positive integer $t \leq 2\varPhi_{p}(\ell,\mathcal{U})$ such that $x|\sob{\mathcal{U}} = (T^{t\vec{v}_{\ell}}x)|\sob{\mathcal{U}}$. If we also have $x|\sob{\mathcal{U} \cup F} = (T^{t\vec{v}_{\ell}}x)|\sob{\mathcal{U} \cup F}$, since $\mathcal{U} \in \mathcal{F}^{V\!ol}_{C}$ is an $(\eta,\ell)$-generating set, by induction, we obtain $$x|\sob{\ell_{\scriptscriptstyle\mathcal{U}} \cup F} = (T^{t\vec{v}_{\ell}}x)|\sob{\ell_{\scriptscriptstyle\mathcal{U}} \cup F}.$$ So to finish this case, it is enough to show that $x|\sob{F}$ is periodic of period $t\vec{v}_{\ell}$. Indeed, let $\xi = (\xi_i)_{i \in \bb{Z}}$ be the sequence defined by $\xi_{i} := (T^{i\,\vec{v}_{\ell}}x)|\sob{\mathcal{I}^{\ell,p_x}(\mathcal{U})}$ for all $i \in \bb{Z}$. As $|A^{\ell,p_x}(\mathcal{U},x)| > 1$ and  $P_{\xi}(p_x) \leq$ $|L^{\ell}(\mathcal{U} \backslash \ell_{\scriptscriptstyle\mathcal{U}},x)| \leq p_x+|A^{\ell,p_x}(\mathcal{U},x)|-2$, let $1 < p_0 \leq p_x$ be the smallest integer such that $P_{\xi}(p_0) \leq p_0 + |A^{\ell,p_0}(\mathcal{U},x)|-2$. It is easy to see that by minimality $P_{\xi}(p_0) = P_{\xi}(p_0-1)$, which means that a word of $p_0-1$ symbols admits exactly one extension to a word of $p_0$ symbols. Hence, since from $x|\sob{\mathcal{U}} = $ $(T^{t\vec{v}_{\ell}}x)|\sob{\mathcal{U}}$ one has $\xi_{0}\xi_{1} \cdots \xi_{p_0-1} = \xi_{t}\xi_{t+1} \cdots \xi_{t+p_0-1}$, by induction it follows that the sequence $\xi$ is periodic of period $t$. In other words, $x|\sob{F}$ is periodic of period $t\vec{v}_{\ell}$.

Suppose $x \in \mathcal{M}(\ell,\mathcal{U})$ and $|A^{\ell,p}(\mathcal{U},x)| = 1$. In this case, there exists a unique  $\mathcal{U} \backslash \ell_{\scriptscriptstyle\mathcal{U}}$-configuration $\gamma \in L^{\ell}(\mathcal{U} \backslash \ell_{\scriptscriptstyle\mathcal{U}},x)$ and, therefore, $$\big|L^{\ell}(\mathcal{U},x)\big|-1 \leq N_{\mathcal{U}}(\ell,\gamma)-1 \leq P_{\eta}(\mathcal{U})-P_{\eta}(\mathcal{U} \backslash \ell_{\scriptscriptstyle\mathcal{U}}).$$ As before, we can assume that there is a positive integer $t \leq P_{\eta}(\mathcal{U})-P_{\eta}(\mathcal{U} \backslash \ell_{\scriptscriptstyle\mathcal{U}})+1$ such that $x|\sob{\mathcal{U}} = (T^{t\vec{v}_{\ell}}x)|\sob{\mathcal{U}}$. Since $|A^{\ell,p}(\mathcal{U},x)| = 1$, $x|\sob{F}$ is in particular periodic of period $t\vec{v}_{\ell}$. The same argument as in the previous case allows us to conclude that the restriction of $x$ to $\ell_{\scriptscriptstyle\mathcal{U}} \cup F$ is periodic of period $t\vec{v}_{\ell}$. Finally, as $P_{\eta}(\mathcal{U}) > P_{\eta}(\mathcal{U} \backslash \ell_{\scriptscriptstyle\mathcal{U}}) $, note that $$P_{\eta}(\mathcal{U})-P_{\eta}(\mathcal{U} \backslash \ell_{\scriptscriptstyle\mathcal{U}})+1 \leq 2\big(P_{\eta}(\mathcal{U})-P_{\eta}(\mathcal{U} \backslash \ell_{\scriptscriptstyle\mathcal{U}})\big) = 2\varPhi_{p}(\ell,\mathcal{U}).$$ This completes the proof.
\end{proof}

For $(\ell,\mathcal{U},p)$-half-strips $F^+(a),F^-(a) \subset \bb{Z}^2$, we define 
\begin{equation}\label{semifaixat1}
(\ell_{\scriptscriptstyle\mathcal{U}} \cup F^+)(a) := \bigcup_{t \geq a} \Big(\mathcal{I}^{\ell,p}(\mathcal{U}) \cup \{i_{\mathcal{U}}(\ell_{\scriptscriptstyle\mathcal{U}})\}+t\vec{v}_{\ell}\Big)
\end{equation}
and
\begin{equation}\label{semifaixat2}
(\ell_{\scriptscriptstyle\mathcal{U}} \cup F^-)(a) := \bigcup_{t \geq a} \Big(\mathcal{F}^{\ell,p}(\mathcal{U}) \cup \{f_{\mathcal{U}}(\ell_{\scriptscriptstyle\mathcal{U}})\}-t\vec{v}_{\ell}\Big).
\end{equation}

\begin{lemma}\label{rem_equival_lema}
For $\eta \in A^{\bb{Z}^2}$ and a rational oriented line $\ell \in \bb{G}_1$, let $\mathcal{U} \in \mathcal{F}^{V\!ol}_{C}$ be an $(\eta,\ell)$-generating set. If $x \not\in$ $\mathcal{M}_{a\pm}(\ell,\mathcal{U})$, then, for every $(\ell,\mathcal{U},p)$-half-strip $F^{\pm}(a)$ with $\conv(F^{\pm}(a)) \cap \bb{Z}^2 = F^{\pm}(a)$ and any configuration $y \in X_{\eta}$, $x|\sob{F^{\pm}(a)} = y|\sob{F^{\pm}(a)}$ implies $x|\sob{(\ell_{\scriptscriptstyle\mathcal{U}} \cup F^{\pm})(a)} = y|\sob{(\ell_{\scriptscriptstyle\mathcal{U}} \cup F^{\pm})(a)}$.
\end{lemma}
\begin{proof}
The proof is identical to that one of Lemma \ref{lem_exist_x}.
\end{proof}

Next lemma is the analogous of Lemma \ref{lema_extens} for half strips (recall that $-\vec{v}_{\ell} = \vec{v}_{\lle}$).

\begin{lemma}\label{lem_extens_semifaixa}
Given $\eta \in A^{\bb{Z}^2}$, suppose $\ell \in \bb{G}_1$ is a one-sided nonexpansive direction on $X_{\eta}$ and $\mathcal{U} \in \mathcal{F}^{V\!ol}_{C}$ is an $(\ell,p)$-balanced set. If the restriction of $x \in X_{\eta}$ to the $(\ell,\mathcal{U},p)$-half-strip $F^{\pm}(a)$, $a \in \bb{Z}$, is periodic of period $\pm t'\vec{v}_{\ell}$ for some $t' \in \bb{N}$, then $x|\sob{(\ell_{\scriptscriptstyle\mathcal{U}} \cup F^{\pm})(a)}$ is periodic of period $\pm t\vec{v}_{\ell}$, where $t = t'$ if $x \not \in \mathcal{M}_{a\pm}(\ell,\mathcal{U})$ and $t \leq$ $2\varPhi_{p}(\ell,\mathcal{U})$ otherwise.
\end{lemma}
\begin{proof}
The proof is similar to that one of Lemma \ref{lema_extens}. If $x \not \in \mathcal{M}_{a\pm}(\ell,\mathcal{U})$, we use Lemma~\ref{rem_equival_lema} instead of Lemma \ref{lem_exist_x}. If $x \in \mathcal{M}_{a\pm}(\ell,\mathcal{U})$ and $|A_{a\pm}^{\ell,p}(\mathcal{U},x)| > 1$, since the sequence $(T^{\pm t\vec{v}_{\ell}}x)_{t \geq a} \subset \mathcal{M}_{a\pm}(\ell,\mathcal{U})$ has an accumulation point $z \in \mathcal{M}(\ell,\mathcal{U})$ with $|A^{\ell,p}(\mathcal{U},z)| > 1$, then there exists a smallest integer $p_{z} \leq p$ fulfilling Definition \ref{dfn_lp_balanceado}. Due to periodicity of $x|\sob{F^{\pm}(a)}$, we have $A^{\ell,p_z}(\mathcal{U},z) = A_{a\pm}^{\ell,p_z}(\mathcal{U},x)$. So just use (\ref{eq_card_a}) instead of  (\ref{desig_card}).
\end{proof}

\begin{proposition}\label{pps_period_semiplano}
Let $\eta \in A^{\bb{Z}^2}$ and suppose $\ell \in \bb{G}_1$ is a one-sided nonexpansive direction on $X_{\eta}$ and $\mathcal{U} \in$ $\mathcal{F}^{V\!ol}_{C}$ is an $(\ell,p)$-balanced set. If the restriction of $x \in X_{\eta}$ to the $(\ell,\mathcal{U},p)$-strip is periodic of period $t'\vec{v}_{\ell}$ for some $t' \in \bb{N}$, then, for any transla\-tion $\mathcal{U}'$ of $\mathcal{U}$ with $\mathcal{U}' \subset \mathcal{H}(\lle_{\scriptscriptstyle\mathcal{U}})$, the restriction of $x$ to the $(\ell,\mathcal{U}',p)$-strip is periodic of period $t\vec{v}_{\ell}$, where $t \leq \max\{t',2\varPhi_{p}(\ell,\mathcal{U})\}$. In particular, the configuration $x|\sob{\mathcal{H}(\lle_{\scriptscriptstyle\mathcal{U}})}$ is $\ell$-periodic.
\end{proposition}
\begin{proof}
Let $u \in (\bb{Z}^2)^*$ be such that $\ell+u = \ell^{(-)}$ (Recall Notation \ref{not_nextline}). Let $P_{\mathcal{U}}$ denote the set of $\kappa \in$ $\bb{Z}_+$ such that the restriction of $x \in X_{\eta}$ to the $(\ell,\mathcal{U}+\kappa u,p)$-strip is periodic of period $\tau'\vec{v}_{\ell}$ for some $\tau' \leq$ $\max\{t',2\varPhi_{\ell,p}(\mathcal{U})\}$. Suppose, by contradiction, that $P_{\mathcal{U}}$ does not coincide with $\bb{Z}_+$. Let $\kappa' \in$ $P_{\mathcal{U}}$ be the largest integer for which $i \in P_{\mathcal{U}}$ for every $0 \leq i \leq \kappa'$. Since the restriction of $x$ to the $(\ell,\mathcal{U}+\kappa'u,p)$-strip $F_{\kappa'}$ is periodic of period $\tau'\vec{v}_{\ell}$ for some integer $\tau' \leq \max\{t',2\varPhi_{p}(\ell,\mathcal{U})\}$, according to Lemma~\ref{lema_extens} the configuration $x|\sob{\ell_{(\mathcal{U}+\kappa'u)} \cup F_{\kappa'}}$ is periodic of period $\tau\vec{v}_{\ell}$, where $\tau = \tau'$ if $x \not \in$ $\mathcal{M}(\ell,\mathcal{U}+\kappa'u)$ and $\tau \leq 2\varPhi_{p}(\ell,\mathcal{U}+\kappa'u)$ otherwise. The fact that $x \in \mathcal{M}(\ell,\mathcal{U}+\kappa'u)$ if, and only if, $T^{\kappa'u}x \in \mathcal{M}(\ell,\mathcal{U})$ implies $\varPhi_{p}(\ell,\mathcal{U}+\kappa'u) = \varPhi_{p}(\ell,\mathcal{U})$. So we have that the restriction of $x$ to the $(\ell,\mathcal{U}+(\kappa'+1)u,p)$-strip is periodic of period $\tau \vec{v}_{\ell}$ with $\tau \leq \max\{t',2\varPhi_{p}(\ell,\mathcal{U})\}$, which contradicts the maximality of $\kappa' \in \bb{Z}_+$.
\end{proof}

The thesis of corollary below was first obtained in Proposition \ref{pps_par_antipar_exp} for periodic configurations. Here, as a consequence of the previous results this hypothesis can be replaced by the existence of a balanced set.

\begin{corollary}\label{cor_reta_antiparalela}
If $\mathcal{U} \in \mathcal{F}^{V\!ol}_{C}$ is $(\ell,p)$-balanced, then the oriented line $\lle$, antiparallel to $\ell  \in \bb{G}_1$, is also a one-sided nonexpansive direction on $X_{\eta}$. 
\end{corollary}
\begin{proof}
Let $x,x' \in X_{\eta}$ be configurations where $x|\sob{\mathcal{H}(\ell)} = x'|\sob{\mathcal{H}(\ell)}$, but $x_g \neq$ $x'_g$ for some $g \in \ell^{(-)} \cap \bb{Z}^2$. Let $\mathcal{U}' \in \mathcal{F}^{V\!ol}_{C}$ denote a translation of $\mathcal{U} \in \mathcal{F}^{V\!ol}_{C}$ such that $\ell_{\scriptscriptstyle\mathcal{U}'} =$ $\ell^{(-)}$. It follows by Lemma \ref{lem_exist_x} that $x,x' \in \mathcal{M}(\ell,\mathcal{U}')$. Thanks to Lemma~\ref{lem_exist_faixa_period}, the restrictions of $x$ and $x'$ to the $(\ell,\mathcal{U}'\!,p)$-strip are $\ell$-periodic. Proposition~\ref{pps_period_semiplano} ensures then that the restrictions of $x$ and $x'$ to the half plane $\mathcal{H}(\lle_{\scriptscriptstyle\mathcal{U}'})$ are $\ell$-peri\-odic of periods $h,h' \in (\bb{Z}^2)^*$. Suppose, by contradiction, that $\lle \in \bb{G}_1$ is a one-sided expansive direction on $X_{\eta}$. Since $x|\sob{\mathcal{H}(\lle)} = (T^{h}x)|\sob{\mathcal{H}(\lle)}$ and $x'|\sob{\mathcal{H}(\lle)} = (T^{h'}\!x')|\sob{\mathcal{H}(\lle)}$, it follows by expansiveness that $x,x' \in X_{\eta}$ are $\ell$-periodic. Thus, due to Proposition~\ref{pps_par_antipar_exp}, the oriented lines $\ell, \lle \in \bb{G}_1$ are both one-sided expansive directions on both sub\-shifts $\overline{Orb\,(x)}$ and $\overline{Orb\,(x')}$. With respect to the others lines, from Lemma \ref{pps_perio_expan} we get that they are also expansive on both sub\-shifts. Hence, Corollary~\ref{cor_2_periodic} implies that $x,x' \in X_{\eta}$ are doubly periodic. Therefore, as $x|\sob{\mathcal{H}(\ell)} = x'|\sob{\mathcal{H}(\ell)}$, we conclude that $x$ and $x'$ are equals, which is a contradiction.
\end{proof}

Note that, by the very definition, the vertices of an edge of an $(\ell,p)$-balanced set parallel to $\lle$ are not necessarily generated. Thus, it is not a surprise that, to extend periodicity to the entire plane, one has to simultaneously consider $(\ell,p)$-balanced and $(\lle,q)$-balanced sets. The following result shows how balanced sets impose peri- odicity for some configurations.

\begin{proposition}\label{pps_period_global}
For $\eta \in A^{\bb{Z}^2}$, suppose $\ell \in \bb{G}_1$ is a one-sided nonexpansive direction on $X_{\eta}$ and $\mathcal{U},\mathcal{T} \in \mathcal{F}^{V\!ol}_{C}$ are, respectively, $(\ell,p)$-balanced and $(\lle,q)$-balanced. If the restriction of $x \in X_{\eta}$ to the $(\ell,\mathcal{U},p)$-strip is periodic of period $t'\vec{v}_{\ell}$ for some $t' \in \bb{N}$, then, for any translation $\mathcal{U}'$ of $\mathcal{U}$, the restriction of $x$ to the $(\ell,\mathcal{U}',p)$-strip is periodic of period $t\vec{v}_{\ell}$, where $t \leq \max\{t',2\varPhi_{p}(\ell,\mathcal{U})\}$. In particular, the configuration $x$ is $\ell$-periodic. Similarly, if the restriction of a  configuration to the $(\lle,\mathcal{T},q)$-strip is periodic, then the configuration is $\lle$-periodic.
\end{proposition}
\begin{proof}
Initially, it follows from Proposition \ref{pps_period_semiplano} applied to each one of the sets $\mathcal{U},\mathcal{T}' \in$ $\mathcal{F}^{V\!ol}_{C}$ that $x$ is $\ell$-periodic, where $\mathcal{T}'$ is a translation of $\mathcal{T}$ with $\mathcal{T}' \subset \mathcal{H}(\lle_{\scriptscriptstyle\mathcal{U}})$. So let $u \in$ $(\bb{Z}^2)^*$ be such that $\ell^{(-)}+u = \ell$. Let $Q_{\mathcal{U}}$ be the set of $\kappa \in \bb{Z}_+$ such that the restriction of $x \in X_{\eta}$ to the $(\ell,\mathcal{U}+\kappa u,p)$-strip is periodic of period $\tau\vec{v}_{\ell}$ for some $\tau \leq \max\{t',2\varPhi_{p}(\ell,\mathcal{U})\}$. Suppose, by contradiction, that $Q_{\mathcal{U}}$ does not coincide with $\bb{Z}_+$. Let $\kappa' \in Q_{\mathcal{U}}$ denote the largest integer for which $i \in Q_{\mathcal{U}}$ for every $0 \leq i \leq \kappa'$. Thus, Lemma \ref{lem_exist_faixa_period} and Proposition~\ref{pps_period_semiplano} imply that $x \not\in \mathcal{M}(\ell,\mathcal{U}+\kappa u)$ for all integer $\kappa > \kappa'$. Since $x$ is $\ell$-periodic, the Pigeonhole Principle ensures that there are integers $i > I \geq \kappa'$ such that $x|\sob{F+iu} = x|\sob{F+Iu}$, where $F \subset \bb{Z}^2$ denotes the $(\ell,\mathcal{U},p)$-strip. We can further assume that $I \geq \kappa'$ is the smallest integer with this property. Since $x|\sob{F+iu} = x|\sob{F+Iu} = (T^{(I-i)u}x)|\sob{F+iu}$ and $x \not\in \mathcal{M}(\ell,\mathcal{U}+iu)$, from Lemma~\ref{lem_exist_x}, we get $$x|\sob{(F+iu) \cup \ell_{(\mathcal{U}+iu)}} = (T^{(I-i)u}x)|\sob{(F+iu) \cup \ell_{(\mathcal{U}+iu)}}.$$ Since $F+(i-1)u \subset \ell_{(\mathcal{U}+iu)} \cup (F+iu)$, one has $x|\sob{F+(i-1)u} = x|\sob{F+(I-1)u}$ and so $I = \kappa'$ (by minimality of $I$). Hence, the restriction of $x$ to the $(\ell,\mathcal{U}+iu,p)$-strip is periodic of period $\tau\vec{v}_{\ell}$ for some $\tau \leq \max\{t',2\varPhi_{p}(\ell,\mathcal{U})\}$. Therefore, by applying Proposition \ref{pps_period_semiplano}, one contradicts the maximality of $\kappa' \in Q_{\mathcal{U}}$.
\end{proof}

\section{Minimal lower complexity generating sets}
\label{sec5}

In this section, we highlight a class of $\eta$-generating sets whose existence comes naturally from a strong condition on the complexity. Although it is a concept deri- ved immediately from our key hypothesis, whose fundamental properties could be absorbed in technical arguments, our main goal in introducing this notion is to sim- plify the statements of results and the expositions of proofs that will follow.

\begin{definition}\label{dfn_mlc_generating}
Given $\eta \in A^{\bb{Z}^2}$, an $\eta$-generating set $\mathcal{S} \in \mathcal{F}_{C}$ is said to be a minimal lower complexity (mlc) $\eta$-generating set if	
\begin{enumerate}
	\item[(i)] $P_{\eta}(\mathcal{S}) \leq \frac{1}{2}|\mathcal{S}|+|A|-1$,
	\item[(ii)] if $\mathcal{T} \subsetneq \mathcal{S}$ is convex and nonempty, then $P_{\eta}(\mathcal{T}) > \frac{1}{2}|\mathcal{T}|+|A|-1$.
\end{enumerate}
\end{definition}

The existence of mlc $\eta$-generating sets is a straightforward consequence of our alphabetical bound assumption for complexity. As a matter of fact, if there exists a set $\mathcal{U} \in \mathcal{F}_{C}$ such that $P_{\eta}(\mathcal{U}) \leq \frac{1}{2}|\mathcal{U}|+|A|-1$, then any convex set $\mathcal{S} \subset \mathcal{U}$ that is minimal among all convex sets $\mathcal{T} \subset \mathcal{U}$ fulfilling $P_{\eta}(\mathcal{T}) \leq \frac{1}{2}|\mathcal{T}|+|A|-1$ is an mlc $\eta$-generating set. The fact that $\frac{1}{2}+|A|-1 < |A| = P_{\eta}(\{g\})$ for all $g \in \bb{Z}^2$ ensures that $\mathcal{S}$ has at least two points. In particular, if $\eta$ is an aperiodic configuration, since $\frac{1}{2}|\mathcal{S}|+|A|-1 \leq |\mathcal{S}|+|A|-2$, from Corollary \ref{lem_per_convnulo}, we get $\mathcal{S} \in \mathcal{F}^{V\!ol}_{C}$.

Due to condition (ii) of Definition \ref{dfn_mlc_generating}, if $\mathcal{T} \subsetneq \mathcal{S}$ is convex and nonempty, it is clear that $P_{\eta}(\mathcal{S}) - P_{\eta}(\mathcal{T}) < \frac{1}{2}|\mathcal{S} \backslash \mathcal{T}|$, which yields
\begin{equation}\label{eq_mlc-generating}
P_{\eta}(\mathcal{S})-P_{\eta}(\mathcal{T}) \leq \Big\lceil \frac{1}{2}\big|\mathcal{S} \backslash \mathcal{T}\big| \Big\rceil-1,
\end{equation}
where as usual $\lceil \, \cdot \, \rceil$ denotes the ceiling function.

\begin{lemma}\label{lem_cardinal_mlc}
For $\eta \in A^{\bb{Z}^2}$,  if $\mathcal{S} \in \mathcal{F}^{V\!ol}_{C}$ is an mlc $\eta$-generating set, then $|\ell_{\scriptscriptstyle\mathcal{S}} \cap \mathcal{S}| \geq 3$ whenever $\ell \in$ $\bb{G}_1$ is a one-sided nonexpansive direction on $X_{\eta}$.
\end{lemma}
\begin{proof}
From (\ref{desig_card}), one has $P_{\eta}(\mathcal{S} \backslash \ell_{\scriptscriptstyle\mathcal{S}}) < P_{\eta}(\mathcal{S})$ and so (\ref{eq_mlc-generating}) applied to $\mathcal{T} = \mathcal{S} \backslash \ell_{\scriptscriptstyle\mathcal{S}}$ pro\-vides $|\ell_{\scriptscriptstyle\mathcal{S}} \cap \mathcal{S}| \geq 3$.
\end{proof}

\begin{remark}\label{obs_exist_arestas_mcbc}
Given $\eta \in A^{\bb{Z}^2}$ aperiodic, suppose $P_{\eta}(\mathcal{U}) \leq \frac{1}{2}|\mathcal{U}|+|A|-1$ for some quasi-regular set $\mathcal{U} \in \mathcal{F}^{V\!ol}_{C}$. If $\ell \in \bb{G}_1$ is a nonexpansive line on $X_{\eta}$, Proposition~\ref{pps_lp-balanceado} and Corollary~\ref{cor_reta_antiparalela} imply that the oriented lines $\ell, \lle \in \bb{G}_1$ are both one-sided non\-expansive directions on $X_{\eta}$. In particular,  if $\mathcal{S} \in \mathcal{F}^{V\!ol}_{C}$ is an mlc $\eta$-generating set, then $\mathcal{S}$ is $(\ell,p)$-balanced with $p = |\ell_{\scriptscriptstyle\mathcal{S}} \cap \mathcal{S}|-1$ or $(\lle,q)$-balanced with $q = |\lle_{\scriptscriptstyle\mathcal{S}} \cap \mathcal{S}|-1$, since, clearly, conditions (i) and (ii) of Proposition \ref{pps_lp-balanceado} hold in one of the two cases.
\end{remark}

Let $\ell \in \bb{G}_1$ be a one-sided nonexpansive direction on $X_{\eta}$ and suppose $\mathcal{S} \in \mathcal{F}^{V\!ol}_{C}$ is an mlc $\eta$-generating set such that $|\ell_{\scriptscriptstyle\mathcal{S}} \cap \mathcal{S}| \leq |\lle_{\scriptscriptstyle\mathcal{S}} \cap \mathcal{S}|$ -- in particular, an $(\ell,p)$-balan\-ced set with $p := |\ell_{\scriptscriptstyle\mathcal{S}} \cap \mathcal{S}|-1$. Since $1 \leq |A^{\ell,p}(\mathcal{S},x)| \leq$ $\lceil \frac{1}{2}|\ell_{\scriptscriptstyle\mathcal{S}} \cap \mathcal{S}| \rceil-1$ for every $x \in \mathcal{M}(\ell,\mathcal{S})$ (see (\ref{desig_card})) and $\lceil \frac{1}{2}|\ell_{\scriptscriptstyle\mathcal{S}} \cap \mathcal{S}| \rceil \leq p$, for each $x \in \mathcal{M}(\ell,\mathcal{S})$, there is a positive integer $p_x \leq p$ such that
\begin{equation}\label{eq_iguald_baixa_compl}
\Big\lceil \frac{1}{2}\big|\ell_{\scriptscriptstyle\mathcal{S}} \cap \mathcal{S}\big| \Big\rceil-1 = p_x+\big|A^{\ell,p}(\mathcal{S},x)\big|-2,
\end{equation}
which yields $P_{\eta}(\mathcal{S})-P_{\eta}(\mathcal{S} \backslash \ell_{\scriptscriptstyle\mathcal{S}}) \leq p_x+|A^{\ell,p_x}(\mathcal{S},x)|-2$. Moreover, from (\ref{eq_iguald_baixa_compl}) it fol\-lows that
\begin{equation}\label{maineq_sec4}
2\varPhi_{p}(\ell,\mathcal{S}) \leq 2\Big\lceil \frac{1}{2}\big|\ell_{\scriptscriptstyle\mathcal{S}} \cap \mathcal{S}\big| \Big\rceil-2 \leq \big|\ell_{\scriptscriptstyle\mathcal{S}} \cap \mathcal{S}\big|-1.
\end{equation}

From the previous discussion, Lemmas \ref{lem_exist_sfaixa_period} and \ref{lem_extens_semifaixa} can be restated as follows. Recall that, for a rational oriented line $\ell$, the vector $\vec{v}_{\lle} = -\vec{v}_{\ell}$ (of minimal norm) is antiparallel to $\ell$ and parallel to $\lle$.  

\begin{proposition}\label{pps_cap3_baixa_compl}
Given $\eta \in A^{\bb{Z}^2}$,  if there exists an mlc $\eta$-generating set $\mathcal{S} \in \mathcal{F}^{V\!ol}_{C}$ and $\ell \in \bb{G}_1$ is a one-sided nonexpansive direction on $X_{\eta}$ with $|\ell_{\scriptscriptstyle\mathcal{S}} \cap \mathcal{S}| \leq |\lle_{\scriptscriptstyle\mathcal{S}} \cap \mathcal{S}|$, then $\mathcal{S}$ is an $(\ell,p)$-balanced set with $p = |\ell_{\scriptscriptstyle\mathcal{S}} \cap \mathcal{S}|-1$ and satisfies the following conditions.
\begin{enumerate}
	\item[(i)] For $x \in \mathcal{M}_{a\pm}(\ell,\mathcal{S})$, with $a \in \bb{Z}$, the restriction of $x$ to the $(\ell,\mathcal{S},p)$-half-strip $F^{\pm}(a+p)$ is periodic of period $\pm t\vec{v}_{\ell}$ for some integer $t \leq \lceil \frac{1}{2}|\ell_{\scriptscriptstyle\mathcal{S}} \cap \mathcal{S}| \rceil-1$;
	\item[(ii)]	If the restriction of $x \in X_{\eta}$ to the $(\ell,\mathcal{S},p)$-half-strip $F^{\pm}(a)$ is periodic of period $\pm t'\vec{v}_{\ell}$, with $t' \in$ $\bb{N}$, then the restriction of $x$ to $(\ell_{\scriptscriptstyle\mathcal{S}} \cup F^{\pm})(a)$ is periodic of period $\pm t\vec{v}_{\ell}$, where $t = t'$ if $x \not\in \mathcal{M}_{a\pm}(\ell,\mathcal{S})$ and $t \leq$ $2\lceil \frac{1}{2}|\ell_{\scriptscriptstyle\mathcal{S}} \cap \mathcal{S}| \rceil-2$ otherwise.
\end{enumerate}
\end{proposition}

\section{$(\ell,\ell')$-periodic maximal $\mathcal{K}$-configurations}
\label{sec6}

If $\eta \in A^{\bb{Z}^2}$ is an aperiodic configuration with $P_{\eta}(n,k) \leq \frac{1}{2}nk$ for some $n,k \in \bb{N}$, in \cite{van} the authors proved  that there always exists another aperiodic configuration $\varphi \in A^{\bb{Z}^2}$ whose restriction to a largest convex set $\mathcal{K}$ is doubly periodic, say, both $\ell$-periodic and $\ell'$-periodic. We will prove the same fact in the alphabetical context (see Proposition \ref{pps_exist_maximal_conf}). As the  proof is quite similar to the one of Section 5.2 of \cite{van}, we will give only the main steps. Such an aperiodic configuration $ \varphi $ is called an $(\ell,\ell')$-periodic maximal $\mathcal{K}$-configuration; see the next definition for accuracy. 

Roughly speaking, for an mlc $\eta$-generating set, its complexity is bounded from above by the sum of the complexity of any proper convex subset and the cardinality of the complement of this subset, as one may already see in (\ref{eq_mlc-generating}). The key point in the proof of Theorem \ref{main_result} is to contradict such a property by taking into account a convenient proper convex subset. In order to estimate the complexity of this subset with respect to $\varphi$, it is useful to relate it to the complexity of a suitable small part. This essentially can be done due to the fact that the periods of an $(\ell,\ell')$-peri- odic maximal $\mathcal{K}$-configuration $\varphi $ can be bounded by expressions depending on the cardinality of the edges of an mlc $\eta$-generating set (see Proposition \ref{afm_principal}).

\begin{definition}\label{ll'maximalperconf}
Let $\ell, \ell' \in \bb{G}_1$ be rational oriented lines, with $\ell \cap \ell' = \{0\}$, such that $\vec{v}_{\ell'} \in \mathcal{H}(\ell)$. Let $\mathcal{K} \subset$ $\bb{Z}^2$ be a convex set with semi-infinite edges parallel to $\ell$ and $\ell'$. An aperiodic configuration $\varphi \in A^{\bb{Z}^2}$ is said to be an $(\ell,\ell'\!)$-periodic maximal $\mathcal{K}$-con- figuration if (i) $\varphi|\sob{\mathcal{K}}$ is doubly periodic with periods in $\ell \cap (\bb{Z}^2)^*$ and $\ell' \cap (\bb{Z}^2)^*$, and (ii) $\mathcal{K}$ is maximal among all convex sets (positively oriented) with semi-infinite edges parallel to $\ell$ and $\ell'$ that satisfy condition (i).
\end{definition}

In the above definition, $\vec{v}_{\ell'} \in \mathcal{H}(\ell)$ means that (with respect to the orientation inhered from the boundary of $\conv \, (\mathcal{K})$) the edge parallel to $\ell$ precedes the edge parallel to $\ell'$. In particular, only $(\ell,\mathcal{S},p)$-half-strips of the form $F^-(a)$, $a \in \bb{Z}$, and $(\ell'\!,\mathcal{S}'\!,q')$-half-strips of the form $F^+(a')$, $a' \in \bb{Z}$, may be contained in $\mathcal{K}$ (see (\ref{def_semi_faixa}) to recall the notion of half strips).

From now on, whenever we mention an $(\ell,\ell'\!)$-periodic maximal $\mathcal{K}$-configuration, we are supposing that $\ell, \ell' \in \bb{G}_1$ are rational oriented lines with $\ell \cap \ell' = \{0\}$ an $\vec{v}_{\ell'} \in \mathcal{H}(\ell)$, as well as that $\mathcal{K} $ is a convex set with semi-infinite edges parallel to $\ell$ and $\ell'$.

Note that, if $\eta \in A^{\bb{Z}^2}$ is an $(\ell,\ell')$-periodic maximal $\mathcal{K}$-configuration, the maximality of $\mathcal{K}$ and Lemma \ref{lem_dir_exp} imply that $\ell,\ell' \in \bb{G}_1$ are one-sided nonexpansive directions on $X_{\eta}$.


\begin{definition}
For $\mathcal{U} \in \mathcal{F}^{V\!ol}_{C}$, a convex set $\mathcal{T} \subset \bb{Z}^2$ is said to be weakly $E(\mathcal{U})$-enve\-loped if, for every edge $\varpi \in E(\mathcal{T})$, there exists an edge $w \in E(\mathcal{U})$ parallel to $\varpi$ with $|w \cap \mathcal{U}| \leq |\varpi \cap \mathcal{T}|$. A set $\mathcal{T} \subset \bb{Z}^2$ weakly $E(\mathcal{U})$-enveloped verifying $|E(\mathcal{T})| = |E(\mathcal{U})|$ is said to be $E(\mathcal{U})$-enveloped.
\end{definition}

It is clear that every set $\mathcal{U} \in \mathcal{F}^{V\!ol}_{C}$ is $E(\mathcal{U})$-enveloped.

\begin{proposition}\label{pps_exist_maximal_conf}
Let $\eta \in A^{\bb{Z}^2}$ be aperiodic and suppose there exist a quasi-regular set $\mathcal{U} \in \mathcal{F}^{V\!ol}_{C}$ such that $P_{\eta}(\mathcal{U}) \leq \frac{1}{2}|\mathcal{U}|+|A|-1$. Let $\ell \in \bb{G}_1$ be a one-sided nonexpan- sive direction on $X_{\eta}$ and $\mathcal{S} \in \mathcal{F}^{V\!ol}_{C}$ be an mlc $\eta$-generating set such that $|\ell^{}_{\scriptscriptstyle\mathcal{S}} \cap \mathcal{S}| \leq$ $|\lle^{}_{\scriptscriptstyle\mathcal{S}} \cap \mathcal{S}|$. Then, one may always obtain a rational oriented line  $\ell' \in \bb{G}_1$ and an in- finite convex region $\mathcal{K} \subset \bb{Z}^2$ for which there exists an $(\ell,\ell'\!)$-periodic maximal  $\mathcal{K}$-con- figuration  $\varphi \in X_{\eta}$.
\end{proposition}
\begin{proof}
Writing $p :=  |\ell_{\scriptscriptstyle\mathcal{S}} \cap \mathcal{S}|-1$, we recall that $\mathcal{S}$ is an $(\ell,p)$-balanced. Since $\ell \in \bb{G}_1$ is a one-sided nonexpansive direction on $X_{\eta}$, there exist $x,y \in$ $X_{\eta}$ such that $x|\sob{\mathcal{H}(\ell)} = y|\sob{\mathcal{H}(\ell)}$, but $x_g \neq y_g$ for some $g \in \ell^{(-)} \cap \bb{Z}^2$. Translating $\mathcal{S}$, we can further assume that $(0,0) \in \mathcal{S}$ and that $\ell_{\scriptscriptstyle\mathcal{S}} = \ell^{(-)}$. Thanks to Lemma~\ref{lem_exist_x},  $x,y \in \mathcal{M}(\ell,\mathcal{S})$. So Propositions~\ref{pps_lp-balanceado} and \ref{pps_period_global} and Lemma \ref{lem_exist_faixa_period} imply that $x$ and $y$ are $\ell$-periodic configurations. But the fact that $x_g \neq y_g$ for some $g \in \ell^{(-)} \cap \bb{Z}^2$ prevents both configurations $x|\sob{\mathcal{H}(\ell^{(-)})}$ and $y|\sob{\mathcal{H}(\ell^{(-)})}$ from being doubly periodic. Thus, we can assume that the restriction of $x$ to the half plane $\mathcal{H}(\ell^{(-)})$ is not doubly periodic.

Note that, for every $\ell$-strip $F \supset \mathcal{S} \backslash \ell_{\scriptscriptstyle\mathcal{S}}$ and any finite set $B \subset F$,  there always exists $x' \in X_{\eta}$ such that $x|\sob{B} = x'|\sob{B}$, but $x|\sob{F} \neq x'|\sob{F}$. Indeed, otherwise, as $x|\sob{B} =$ $(T^g\eta)|\sob{B}$ for some $g \in \bb{Z}^2$, the restriction of $T^g\eta$ to the $(\ell,\mathcal{S},p)$-strip would be $\ell$-periodic and then Propositions~\ref{pps_lp-balanceado} and \ref{pps_period_global} would imply that $\eta$ is periodic, contradicting the aperiodicity of $\eta$. This allows us to construct a sequence of convex, $E(\mathcal{S})$-enveloped, finite sets $B_1 \subset A_{1} \subset B_2 \subset A_2 \subset \cdots \subset B_i \subset A_i \subset \cdots$, with $A_0 \in \ell_{\scriptscriptstyle\mathcal{S}} \cap \mathcal{S}$ and $B_{1} := \mathcal{S}$, and configurations $(\varphi_{i})_{i \in \bb{N}} \in$ $X_{\eta}$ such that, for each $i \in \bb{N}$, $B_i \in \mathcal{F}^{V\!ol}_{C}$ is an $E(\mathcal{S})$-enveloped set with $\ell_{B_i} = \ell_{\scriptscriptstyle\mathcal{S}}$, $B_i$ contains both $A_{i-1}$ and $[-i+1,i-1]^2 \cap \mathcal{H}(\ell^{(-)})$, $x|\sob{B_i} = \varphi_i|\sob{B_i}$, but $x|\sob{F_i} \neq \varphi_i|\sob{F_i}$ and $A_i$ is a maximal set among all $E(\mathcal{S})$-enveloped sets $\mathcal{T} \in \mathcal{F}^{V\!ol}_{C}$ such that $B_{i} \subset$ $\mathcal{T} \subset F_{i}$ and $x|\sob{\mathcal{T}} = \varphi_{i}|\sob{\mathcal{T}}$, where $$F_{i} := \left\{g \in \ell'' \cap \bb{Z}^2 : \ell'' \ \textrm{is parallel to} \ \ell \ \textrm{and} \ \ell'' \cap B_i \neq \emptyset\right\}.$$

For each integer $i \geq 1$, we have $\ell_{\scriptscriptstyle\mathcal{S}} \cap \mathcal{S} \subset \ell_{A_i} \cap A_i \subset \ell^{(-)}$. So let $g_1 := f_{A_1}(\ell^{(-)})$ be the final point of $\ell_{A_1} \cap A_1$ (with respect to the orientation of $\ell^{(-)}$) and, for each $i \in \bb{N}$, let $k_i \in$ $\bb{N}$ be such that $f_{A_i-k_i\vec{v}_{\ell}}(\ell^{(-)}) = g_1$.  We define $\hat{A}_i := A_i-k_i\vec{v}_{\ell}$. As $x \in X_{\eta}$ is $\ell$-periodic, there is an integer $k \geq 0$ such that $T^{k\vec{v}_{\ell}}x = T^{k_i\vec{v}_{\ell}}x$ for infinitely many $i$. By passing to a subsequence, we can assume this holds for all $i$.

Suppose $\varpi_i(0),\varpi_i(1), \ldots, \varpi_i(K),\varpi_i(K+1) \in E(\hat{A}_i)$ are enumerated according to the orientation inherited from the boundary of $\conv(\hat{A}_i)$, where $\varpi_i(0)$ is parallel to $\ell$ and $\varpi_i(0)$ and $\varpi_i(K+1)$ are antiparallel. Since $\bigcup_{i=1}^{\infty} A_i = \mathcal{H}(\ell^{(-)})$, let $1 \leq$ $k_{min} \leq K$ be the smallest integer such that $|\varpi_i(k_{min}) \cap \hat{A}_i| < |\varpi_{i+1}(k_{min}) \cap \hat{A}_{i+1}|$ for infinity many $i$. By passing to a subsequence, we can assume that this holds for all $i$ and, if $k_{min} > 1$, that $|\varpi_i(k) \cap \hat{A}_i| = |\varpi_{i+1}(k) \cap \hat{A}_{i+1}|$ for every $1 \leq k \leq k_{min}-1$ and all $i \in \bb{N}$. In particular, $\hat{A}_{\infty} := \bigcup_{i=1}^{\infty} \hat{A}_i$ is a convex weakly $E(\mathcal{S})$-enveloped set, with two semi-infinite edges, one of which is parallel to $\ell \in \bb{G}_1$ and the other one is parallel to the edges $\varpi_i(k_{min})$.


We define $\hat{\varphi}_i := T^{k\vec{v}_{\ell}}\varphi_i$ for all $i$ and $\hat{x} := T^{k\vec{v}_{\ell}}x$. As $\hat{\varphi}_{i'}|\sob{\hat{A}_i} = \hat{\varphi}_i|\sob{\hat{A}_i}$ for all $1 \leq i \leq i'$, by the compactness of $X_{\eta}$, the sequence $\{\hat{\varphi}_i\}_{i \in \bb{N}}$ has an accumulation point $\varphi \in X_{\eta}$ with $\varphi|\sob{\hat{A}_{\infty}} =$ $\hat{x}|\sob{\hat{A}_{\infty}}$. In particular, it follows that $\varphi|\sob{\hat{A}_{\infty}}$ is an $\ell$-perio- dic configuration. 

For each $0 \leq j \leq k_{min}$, let $\ell_{j} \subset \bb{R}^2$ denote the oriented line parallel to the edges $\varpi_i(j) \in$ $E(\hat{A}_i)$ such that $\varpi_i(j) \subset \ell_j$ for all $i$, and let $\wp_{m+1} := \wp^{(-)}_{m}$ for all $m \geq 0$, where $\wp_0 := \ell_{k_{min}}$. Writing $\hat{A}^{(m)}_{\infty} := \mathcal{H}(\ell_{k_{min}-1}) \cap (\hat{A}_{\infty} \cup$ $\wp_{0} \cup \wp_{1} \cup \cdots \cup \wp_{m})$ for all $m \geq 0$, it is not difficult to argue that, from the maximality of $\hat{A}_i$, there exists $n \geq 0$ such that
\begin{equation}\label{eq_existence_tau_lem}
\varphi|\sob{\hat{A}_{\infty}^{(n)}} = \hat{x}|\sob{\hat{A}_{\infty}^{(n)}}, \ \ \textrm{but} \ \  \varphi|\sob{\hat{A}_{\infty}^{(n+1)}} \neq \hat{x}|\sob{\hat{A}_{\infty}^{(n+1)}}.
\end{equation}

Let $\ell' \in \bb{G}_1$ be the oriented line parallel to the edges $\varpi_i(k_{min})$. If $\mathcal{T} \in \mathcal{F}^{V\!ol}_{C}$ is an $(\eta,\ell')$-generating set such that $\ell'_{\scriptscriptstyle\mathcal{T}} = \wp_{n+1}$ and $a \in \bb{Z}$ satisfies $\mathcal{T} \backslash \ell'_{\scriptscriptstyle\mathcal{T}}+a\vec{v}_{\ell'} \subset$ $\hat{A}_{\infty}^{(n)}$, then $\varphi \in \mathcal{M}_{a+}(\ell'\!,\mathcal{T})$. In fact,  if there exists a $\mathcal{T} \backslash \ell'_{\scriptscriptstyle\mathcal{T}}$-configuration $\gamma \in L_{a+}^{\ell'}(\mathcal{T} \backslash \ell'_{\scriptscriptstyle\mathcal{T}},\varphi)$ with $N_{\mathcal{T}}(\ell'\!,\gamma) = 1$, from $\varphi|\sob{\hat{A}_{\infty}^{(n)}} = \hat{x}|\sob{\hat{A}_{\infty}^{(n)}}$,  it follows that
\begin{equation}\label{eq_igualdade_U_balan}
\varphi|\sob{\hat{A}_{\infty}^{(n)} \cup (\mathcal{T}+a\vec{v}_{\ell'})} = \hat{x}|\sob{\hat{A}_{\infty}^{(n)} \cup (\mathcal{T}+a\vec{v}_{\ell'})}.
\end{equation}
Although $\hat{A}_{\infty}^{(n)}$ may be not weakly $E(\mathcal{S})$-enveloped,  both $\hat{A}_{\infty}^{(n)}$ and $\mathcal{S}$ have suitable geometries to fully take advantage of (\ref{eq_igualdade_U_balan}). Therefore, as $\mathcal{S}$ is an $\eta$-generating set, from (\ref{eq_igualdade_U_balan}) it follows by induction that $\varphi|\sob{\hat{A}_{\infty}^{(n+1)}} = \hat{x}|\sob{\hat{A}_{\infty}^{(m+1)}}$, which contradicts (\ref{eq_existence_tau_lem}).

Note that $\ell' \in \bb{G}_1$ is a one-sided nonexpansive direction on $X_{\eta}$, since, other\-wise, from Lemma \ref{lem_dir_exp} we could get (\ref{eq_igualdade_U_balan}) for some integer $b \geq a$ and, exactly as above, reach an absurd. Hence, Proposition \ref{pps_lp-balanceado} provides an $(\ell'\!,q')$-balanced set $\mathcal{T} \in \mathcal{F}^{V\!ol}_{C}$ with $q' := |\ell'_{\scriptscriptstyle\mathcal{T}} \cap \mathcal{T}|-1$, which can be supposed such that $\ell'_{\scriptscriptstyle\mathcal{T}} = \wp_{n+1}$. For $a \in$ $\bb{Z}$ such that the $(\ell'\!,\mathcal{T},q')$-half-strip $F^+(a)$ is contained in $\hat{A}^{(n)}_{\infty}$, the previous discussion ensures that $\varphi \in$ $\mathcal{M}_{a+}(\ell'\!,\mathcal{T})$. Remark \ref{rem_balan_semifaixa1} and Lemma \ref{lem_exist_sfaixa_period} imply that there exists an integer $b \geq a$ such that $\varphi|\sob{F^+(b)}$ is periodic of period $t'\vec{v}_{\ell'}$ for some $t' \in \bb{N}$. As $\varphi|\sob{\hat{A}_{\infty}^{(n)}}$ is periodic of period $t \vec{v}_{\lle}$ for some $t \in \bb{N}$, then $\varphi|\sob{F^+(b)+t \vec{v}_{\lle}}$ is also periodic of period $t'\vec{v}_{\ell'}$. Applying Lemma \ref{lem_extens_semifaixa} successively to fill out the gap between $F^+(b)$ and $F^+(b)+t \vec{v}_{\lle}$ (if there exists), we obtain a set $\mathcal{W} \subset \hat{A}_{\infty}^{(n)}$ such that $\varphi|\sob{\mathcal{W}}$ is both $\ell$-periodic (since $\varphi|\sob{\hat{A}_{\infty}^{(n)}}$ is $\ell$-periodic) and $\ell'$-periodic (see Figure~\ref{ext_corol_2period}). 
\begin{figure}[ht]
	\centering
	\def\svgwidth{4.5cm}
	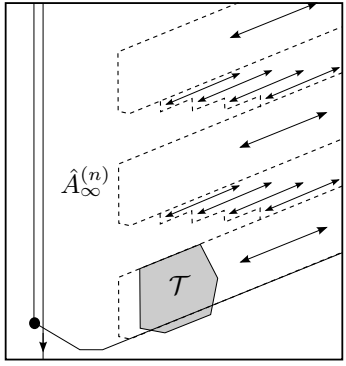
	\caption{The dashed region represents the set $\mathcal{W} \subset \hat{A}^{(n)}_{\infty}$ where $\varphi|\sob{\mathcal{W}}$ is doubly periodic.}
	\label{ext_corol_2period}
\end{figure}
In any case, there exists a convex weakly $E(\mathcal{S})$-enveloped set $\mathcal{K}' \subset \hat{A}_{\infty}^{(n)}$ with two semi-infinite edges parallel, respectively, to $\ell$ and $\ell'$ such that $\varphi|\sob{\mathcal{K}'}$ is $\ell$-periodic and $\ell'$-periodic.

Let $\mathcal{K} \subset \bb{Z}^2$ denote the maximal set among all convex sets $Q \supset \mathcal{K}'$ such that $\varphi|\sob{Q}$ is both $\ell$-periodic and $\ell'$-periodic.

Since $\mathcal{K}$ is convex and $\mathcal{K}'$ has two semi-infinite edges parallel, respectively,  to $\ell$ and $\ell'$ (recall that $\wp_{n}$ is parallel to $\ell'$), the inclusions $\mathcal{K}' \subset \mathcal{K} \subset \mathcal{H}(\ell) \cap \mathcal{H}(\wp_{n})$ ensure that $\mathcal{K}$ has two semi-infinite edges parallel, respectively,  to $\ell$ and $\ell'$. To conclude the proof, it is enough to show that $\mathcal{K} \subset \mathcal{H}(\ell) \cap \mathcal{H}(\wp_{n})$. If $\mathcal{K} \not\subset \mathcal{H}(\wp_{n})$, then $\wp_{n+1} \cap$ $\conv(\mathcal{K})$ is a half line. As $\varphi|\sob{\mathcal{K}}$ and $\hat{x}$ are $\ell$-periodic and $\varphi|\sob{\hat{A}^{(n)}_{\infty}} =  \hat{x}|\sob{\hat{A}^{(n)}_{\infty}}$, it follows that $\varphi|\sob{\hat{A}^{n}_{\infty} \cup (\wp_{n+1} \cap \mathcal{K})} = \hat{x}|\sob{\hat{A}^{n}_{\infty} \cup (\wp_{n+1} \cap \mathcal{K})}$. But such equality is equivalent to (\ref{eq_igualdade_U_balan})  and, by the same reasoning that succeeds it, one contradicts (\ref{eq_existence_tau_lem}). Now, if $\mathcal{K} \not\subset$ $ \mathcal{H}(\ell)$, then $\ell'' \cap \conv(\mathcal{K})$ is a half line for every line $\ell'' \subset$ $\mathcal{H}(\ell^{(-)})$ parallel to $\ell$. This allows us to transfer the doubly periodicity of $\hat{x}|\sob{\mathcal{K} \cap \hat{A}^{(n)}_{\infty} \cap \mathcal{H}(\ell^{(-)})}$ to $\hat{x}|\sob{\mathcal{H}(\ell^{(-)})}$, which is an absurd, since we are assuming $x|\sob{\mathcal{H}(\ell^{(-)})}$ is not doubly periodic.

Finally, from Lemma \ref{lem_dir_exp} and maximality of $\mathcal{K}$, it follows that $\ell,\ell' \in \bb{G}_1$ are one-sided  nonexpansive directions on $\overline{Orb \, (\varphi)}$, which means that $\varphi \in X_{\eta}$ is aperiodic and so an $(\ell,\ell')$-periodic maximal $\mathcal{K}$-configuration.
\end{proof}

If $\mathcal{T} \in \mathcal{F}_{C}$ and $\ell \in \bb{G}_1$, from now on, let $\diam_{\ell}(\mathcal{T})$ denote the number of distinct oriented lines parallel to $\ell$ that have nonempty intersection with $\mathcal{T}$.

\begin{lemma}\label{lemma_hafl-set}
Given $\eta \in A^{\bb{Z}^2}$, if $\mathcal{S} \in \mathcal{F}^{V\!ol}_{C}$ is mlc $\eta$-generating and $\ell,\lle \in \bb{G}_1$ are anti- parallel one-sided nonexpansive directions on $X_{\eta}$, then there exists an $\eta$-generating set $\mathcal{T} \subset \mathcal{S}$ with $\diam_{\ell}(\mathcal{T}) \leq \lceil \frac{1}{2}\diam_{\ell}(\mathcal{S}) \rceil$.
\end{lemma}
\begin{proof}
Initially, choose $\ell_1, \ell_2 \subset \bb{R}^2$ antiparallel oriented lines, with $\ell_1$ parallel to $\ell$, such that $\mathcal{U}_1 := \mathcal{S} \cap \mathcal{H}(\ell_1)$ and $\mathcal{U}_2 := \mathcal{S} \cap \mathcal{H}(\ell_2)$ satisfy $\diam_{\ell}(\mathcal{U}_1) = \lceil \frac{1}{2}\diam_{\ell}(\mathcal{S}) \rceil$ and $\diam_{\ell}(\mathcal{U}_2) = \lceil \frac{1}{2}\diam_{\ell}(\mathcal{S}) \rceil$. Since $\mathcal{U}_1 \subset \mathcal{S} \backslash \ell^{}_{\scriptscriptstyle\mathcal{S}}$, $\mathcal{U}_2 \subset \mathcal{S} \backslash \lle^{}_{\scriptscriptstyle\mathcal{S}}$ and $\ell,\lle \in \bb{G}_1$ are one-sided nonexpansive directions on $X_{\eta}$, one has $P_{\eta}(\mathcal{U}_1) < P_{\eta}(\mathcal{S})$ and $P_{\eta}(\mathcal{U}_2) < P_{\eta}(\mathcal{S})$. Hence, if $|\mathcal{U}_1| \geq \frac{1}{2}|\mathcal{S}|$, then $$P_{\eta}(\mathcal{U}_1)-|\mathcal{U}_1| < P_{\eta}(\mathcal{S})-|\mathcal{U}_1| \leq \frac{1}{2}|\mathcal{S}|+|A|-1-|\mathcal{U}_1| \leq \frac{1}{2}|\mathcal{S}|+|A|-1-\frac{1}{2}|\mathcal{S}| = |A|-1,$$ that is, $P_{\eta}(\mathcal{U}_1) \leq |\mathcal{U}_1|+|A|-2$. Thus, from Remark \ref{obs_exist_generating}, there exists an $\eta$-generating set $\mathcal{T} \subset \mathcal{U}_1$ and, in particular, $\diam_{\ell}(\mathcal{T}) \leq \lceil \frac{1}{2}\diam_{\ell}(\mathcal{S}) \rceil$. If $|\mathcal{U}_1| < \frac{1}{2}|\mathcal{S}|$, it is clear that $|\mathcal{U}_2| \geq \frac{1}{2}|\mathcal{S}|$ and the same argument as before concludes the proof.
\end{proof}

The next proposition generalizes Claim 5.4 in \cite{van}. This result shows that the existence of an mlc generating set provides upper bounds for the periods of restrictions to $\mathcal{K}$ of $(\ell,\ell')$-periodic maximal $\mathcal{K}$-configurations, with one of these periods even remaining outside the convex region $\mathcal{K}$. The strategy of its proof consists in arguing that there exists a suitable translation of the mlc generating set such that each restriction to the associated half strip admits multiple extensions, so that one may apply Proposition \ref{pps_cap3_baixa_compl}. Some notions used here can be recalled in Notation~\ref{not_nextline}, Definition \ref{retabase}, (\ref{def_semi_faixa}) and (\ref{semifaixat2}). Furthermore, recall that $\mathcal{M}_{a+}(\ell,\mathcal{U})$ and $\mathcal{M}_{a-}(\ell,\mathcal{U})$ are the sets formed by the configurations $x \in X_{\eta}$ that satisfy,  respectively, (\ref{eq_pontos_naoexpan_sfaixa_v1}) and (\ref{eq_pontos_naoexpan_sfaixa_v2}).

\begin{proposition}\label{afm_principal}
For $\eta \in A^{\bb{Z}^2}$, suppose there exists an $(\ell,\ell')$-periodic maximal $\mathcal{K}$-configuration $\varphi \in X_{\eta}$. Let $\mathcal{R}' \in \mathcal{F}^{V\!ol}_{C}$ be an $(\lle'\!,r')$-balanced set  and let $\mathcal{S} \in \mathcal{F}^{V\!ol}_{C}$ be an mlc $\eta$-generating set such that $\ell^{}_{\scriptscriptstyle\mathcal{S}} = \ell_{\scriptscriptstyle\mathcal{K}}^{(-)}$ and $|\ell^{}_{\scriptscriptstyle\mathcal{S}} \cap \mathcal{S}| \leq$ $|\lle^{}_{\scriptscriptstyle\mathcal{S}} \cap \mathcal{S}|$. Then the following conditions hold:
\begin{enumerate}
	\item[(i)] $\varphi|\sob{\mathcal{K}}$ is periodic of period $t_0\vec{v}_{\lle}$ for some integer $t_0 \leq \lceil \frac{1}{2}|\ell^{}_{\scriptscriptstyle\mathcal{S}} \cap \mathcal{S}| \rceil-1$ and periodic of period $t'_0\vec{v}_{\ell'}$ for some integer $t_0' \leq |\ell'_{\scriptscriptstyle\mathcal{S}} \cap \mathcal{S}|-2$,
	\item[(ii)] denoting $p := |\ell^{}_{\scriptscriptstyle\mathcal{S}} \cap \mathcal{S}|-1$, for any $(\ell,\mathcal{S},p)$-half-strip $F^-(a)$,  $a \in \bb{Z}$, contained in $\mathcal{K}$,  the restriction of $\varphi$ to $(\ell^{}_{\scriptscriptstyle\mathcal{S}} \cup F^-)(a)$ is periodic of period $\tau_0\vec{v}_{\lle}$ for some integer $\tau_0 \leq$ $\lceil \frac{1}{2}|\ell^{}_{\scriptscriptstyle\mathcal{S}} \cap \mathcal{S}| \rceil-1$.
\end{enumerate}
\end{proposition}
\begin{proof}
Let $h = \kappa \vec{v}_{\lle}$ and $h' = \kappa'\vec{v}_{\ell'}$, with $\kappa,\kappa' \in \bb{N}$, be periods of $\varphi|\sob{\mathcal{K}}$. Recall that $\ell,\ell' \in$ $\bb{G}_1$ are one-sided nonexpansive directions on $\overline{Orb \, (\varphi)}$ and so one-sided nonexpansive directions on $X_{\eta}$. Since $\mathcal{S}$ is $(\ell,p)$-balanced and $(T^{h}\varphi)|\sob{F^-(a)} = \varphi|\sob{F^-(a)} =$ $(T^{h'}\varphi)|\sob{F^-(a)}$,  Lemma \ref{rem_equival_lema} and maximality of $\mathcal{K}$ provide $\varphi \in \mathcal{M}_{a-}(\ell, \mathcal{S})$. Thanks to condition (i) of Proposition~\ref{pps_cap3_baixa_compl} the restriction of $\varphi$ to the $(\ell,\mathcal{S},p)$-half-strip $F^-(a+p)$ is periodic of period $t_0\vec{v}_{\lle}$ for some
\begin{equation}\label{eq_cap3_periodo_1}
t_0 \leq \Big\lceil\frac{1}{2}\big|\ell^{}_{\scriptscriptstyle\mathcal{S}} \cap \mathcal{S}\big|\Big\rceil-1.
\end{equation}
Moreover, since $\varphi|\sob{\mathcal{K}}$ is $\ell$-periodic and $F^-(a+p) \subset F^-(a) \subset \mathcal{K}$, it is clear that the restriction of $\varphi$ to the $(\ell,\mathcal{S},p)$-half-strip $F^-(a)$ is also periodic of period $t_0\vec{v}_{\lle}$.

To conclude the proof of the first condition it is enough to show that $\varphi|\sob{\mathcal{K}}$ is peri- odic of period $t'_0 \vec{v}_{\ell'}$ for some integer $t'_0 \leq |\ell'_{\scriptscriptstyle\mathcal{S}} \cap \mathcal{S}|-2$. Indeed, in this case, $F^-(a) \cap$ $(F^-(a)+t'_0\vec{v}_{\ell'}) \neq \emptyset$ and, from the Fine-Wilf Theorem \cite{fine}, we obtain that $\varphi|\sob{\mathcal{K}}$ is periodic of period $t_0\vec{v}_{\lle}$.

We prove such a fact by considering two cases separately.

\medbreak{\it Case 1.} Suppose $|\ell'_{\scriptscriptstyle\mathcal{S}} \cap \mathcal{S}| \leq |\lle'_{\scriptscriptstyle\mathcal{S}} \cap \mathcal{S}|$ and define $q' := |\ell'_{\scriptscriptstyle\mathcal{S}} \cap \mathcal{S}|-1$. Let $\mathcal{S}'$ be a translation of $\mathcal{S}$ with $\ell'_{\scriptscriptstyle\mathcal{S}'} = \ell'{}^{(-)}_{\hspace{-0.1cm}\scriptscriptstyle\mathcal{K}}$ and let $a' \in \bb{Z}$ be such that the $(\ell'\!,\mathcal{S}'\!,q')$-half-strip $F^+(a')$ is contained in $\mathcal{K}$. As before, Lemma \ref{rem_equival_lema} and maximality of $\mathcal{K}$ imply that $\varphi \in \mathcal{M}_{a'\!+}(\ell'\!,\mathcal{S}')$. Thus, condition (i) of Proposition~\ref{pps_cap3_baixa_compl} ensures that the restriction of $\varphi$ to $F^+(a'+q')$ and therefore to $F^+(a')$ is periodic of period $t'_0 \vec{v}_{\ell'}$ for some integer
\begin{equation}\label{eq_cap3_periodo_2}
t'_0 \leq \Big\lceil\frac{1}{2}\big|\ell'_{\scriptscriptstyle\mathcal{S}'} \cap \mathcal{S}'\big|\Big\rceil-1 = \Big\lceil\frac{1}{2}\big|\ell'_{\scriptscriptstyle\mathcal{S}} \cap \mathcal{S}\big|\Big\rceil-1 \leq \big|\ell'_{\scriptscriptstyle\mathcal{S}} \cap \mathcal{S}\big|-2.
\end{equation}
Clearly, for each $j \geq 0$, $\varphi|\sob{F^-(a)} = (T^{jh'}\varphi)|\sob{F^-(a)}$ and $\varphi|\sob{F^+(a')} = (T^{jh}\varphi)|\sob{F^+(a')}$. Let $m,m' \in \bb{N}$ be such that $$\mathcal{W} := (F^-(a)+m'h') \cap (F^+(a')+mh)$$ contains a translation of $\mathcal{S} \backslash \{\ell^{}_{\scriptscriptstyle\mathcal{S}},\ell'_{\scriptscriptstyle\mathcal{S}}\}$. Let $\lle'_1 := \lle'{}^{(-)}_{\hspace{-0.1cm}\scriptscriptstyle\mathcal{W}}$ and define $\lle'_{i+1} := \lle'_i{}^{(-)}$ for all $i \geq 1$. For each $i \geq 1$, we write $(\bb{Z}^2 \backslash \mathcal{H}(\lle^{}_{\scriptscriptstyle\mathcal{W}})) \cap \lle'_i =: \{g_i+t\vec{v}_{\ell'} : t \geq 0\}$. Thanks to Lemma~\ref{lemma_hafl-set} there exists an $\eta$-generating set $\mathcal{T} \subset \mathcal{S}$ with $\diam_{\ell}(\mathcal{T}) \leq \lceil \frac{1}{2}\diam_{\ell}(\mathcal{S}) \rceil$. If $\mathcal{U}_i$ is the translation of $\mathcal{T}$ where $g_i$ is the initial point of $\lle'_{\scriptscriptstyle\mathcal{U}_i} \cap \mathcal{U}_i$ (with respect to the orientation of $\lle'$), by (\ref{eq_cap3_periodo_2}) we have $\mathcal{U}_i-t'_0\vec{v}_{\ell'} \subset$ $F^-(a)+m'h'$ (see Figure~\ref{ext_afm_pri}). 
\begin{figure}[ht]
	\centering
	\def\svgwidth{4.4cm}
	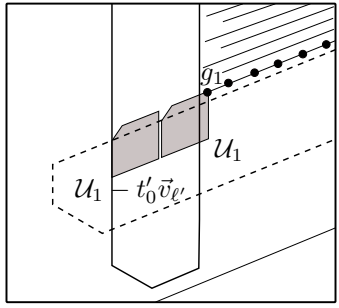
	\caption{The set $\mathcal{U}$. The dashed region represents $F^+(a')+mh$.}
	\label{ext_afm_pri}
\end{figure}
Since $\varphi|\sob{F^-(a)+m'h'}$ is periodic of period $t_0\vec{v}_{\lle}$, from (\ref{eq_cap3_periodo_1}) it follows that the res\-triction of $\varphi$ to the set $$(F^-(a)+m'h') \cup (F^+(a')+mh)$$ is periodic of period $t'_0\vec{v}_{\ell'}$. In particular, $\varphi|\sob{\mathcal{U}_1 \backslash \{g_1\}} = (T^{-t'_0\vec{v}_{\ell'}}\varphi)|\sob{\mathcal{U}_1 \backslash \{g_1\}}$, which yields $\varphi_{g_1} =$ $(T^{-t'_0\vec{v}_{\ell'}}\varphi)_{g_1}$, because $\mathcal{U}_1$ is an $\eta$-generating set. Applying an identical reasoning to the points $g_1+t\vec{v}_{\ell'}$, we obtain by induction that the restriction of $\varphi$ to the set  $$(F^-(a)+m'h') \cup (F^+(a')+mh) \cup \{g_1 +t\vec{v}_{\ell'} : t \geq 0\}$$ is periodic of period $t'_0 \vec{v}_{\ell'}$. By repeating the reasoning for the others lines $\lle'_i$, we conclude by induction that the restriction of $\varphi$ to the set $$(F^-(a)+m'h') \cup (F^+(a')+mh) \cup \left\{g_i +t\vec{v}_{\ell'} : i \geq 1, \ t \geq 0\right\}$$ is periodic of period $t'_0 \vec{v}_{\ell'}$. Since by hypothesis $\varphi|\sob{\mathcal{K}}$ is doubly periodic with periods in $\ell \cap (\bb{Z}^2)^*$ and $\ell' \cap (\bb{Z}^2)^*$, using the Fine-Wilf Theorem it is not difficult to conclude that $\varphi|\sob{\mathcal{K}}$ is also periodic of period $t'_0 \vec{v}_{\ell'}$, which completes this case.

\medbreak{\it Case 2.} Suppose $|\ell'_{\scriptscriptstyle\mathcal{S}} \cap \mathcal{S}| > |\lle'_{\scriptscriptstyle\mathcal{S}} \cap \mathcal{S}|$ and define $q' := |\lle'_{\scriptscriptstyle\mathcal{S}} \cap \mathcal{S}|-1$. Note that, Corollary~\ref{cor_reta_antiparalela} and Lemma \ref{lem_cardinal_mlc} imply $q' \geq 2$. If for some translation $\mathcal{S}'$ of $\mathcal{S}$ there exists $a' \in \bb{Z}$ such that $\varphi \in \mathcal{M}_{a'\!-}(\lle'\!,\mathcal{S}')$ and the $(\lle'\!,\mathcal{S}'\!,q')$-half-strip $F^-(a')$ lies in $\mathcal{K}$, then exactly as argued in Case~1 we have that $\varphi|\sob{\mathcal{K}}$ is periodic of period $t'_0\vec{v}_{\ell'}$ for some $t'_0 \leq \lceil \frac{1}{2}|\lle'_{\scriptscriptstyle\mathcal{S}} \cap \mathcal{S}| \rceil-1 <$ $|\lle'_{\scriptscriptstyle\mathcal{S}} \cap \mathcal{S}|-1 \leq |\ell'_{\scriptscriptstyle\mathcal{S}} \cap \mathcal{S}|-2$, which completes this case. 

Otherwise, $\varphi \not\in \mathcal{M}_{a'\!-}(\lle'\!,\mathcal{S}')$ for all translation $\mathcal{S}'$ of $\mathcal{S}$ and any $a' \in \bb{Z}$ such that the $(\lle'\!,\mathcal{S}'\!,q')$-half-strip $F^-(a')$ is contained in $\mathcal{K}$. By hypothesis, there is an $(\ell'\!,r')$-balanced set $\mathcal{R}' \in \mathcal{F}^{V\!ol}_{C}$, which, without loss of generality, is supposed to satisfy $\ell'_{\scriptscriptstyle\mathcal{R}'} = \ell'{}^{(-)}_{\hspace{-0.1cm}\scriptscriptstyle\mathcal{K}}$. Let $u \in (\bb{Z}^2)^*$ be such that $\ell'+u = \ell'{}^{(-)}$ and consider $\mathcal{R}'_{j} := \mathcal{R}'+(j-1)u$ for all $j \geq 1$. To simplify notions, we write $\ell'_j := \ell'_{\scriptscriptstyle\mathcal{R}'_j}$. Let $b'_1 \in$ $\bb{Z}$ be such that the  $(\ell'\!,\mathcal{R}'_1,r')$-half-strip $F^+_1(b'_1)$  lies in $\mathcal{K}$. Thanks to Lemma~\ref{lem_extens_semifaixa}, the restriction of $\varphi$ to the set $(\ell'_{1} \cup F_1^+)(b'_1)$ is periodic of period $t_1\vec{v}_{\ell'}$ for some $t_1 \leq \max\{\kappa',2\varPhi\}$, where $\varPhi := \varPhi_{r'}(\ell'\!,\mathcal{R}'_j)$ does not depend on $j \in \bb{N}$. Suppose that we have defined a sequence of integers $b'_1, b'_2, \ldots, b'_m$ such that, for each $1 < j \leq m$, $F^+_j(b'_j)$ is an $(\ell'\!,\mathcal{R}'_j,r')$-half-strip that satisfies
\begin{enumerate}
	\item[(i)] $\varphi|\sob{(\ell'_{j} \cup F^+_j)(b'_j)}$ is periodic of period $t_j\vec{v}_{\ell'}$ for some $t_j \leq \max\{\kappa',2\varPhi\}$,
	\item[(ii)] $F^+_j(b'_{j}) \subset (\ell'_{j-1} \cup F^+_{j-1})(b'_{j-1})$.
\end{enumerate}
Let $b'_{m+1} \in$ $\bb{Z}$ be such that the $(\ell'\!,\mathcal{R}'_{m+1},r')$-half-strip $F^+_{m+1}(b'_{m+1})$ is contained in $(\ell'_{m}\! \cup F^+_{m})(b'_m)$. Lemma~\ref{lem_extens_semifaixa} implies that the restriction of $\varphi$ to $(\ell'_{m+1}\! \cup F^+_{m+1})(b'_{m+1})$ is periodic of period $t_{m+1}\vec{v}_{\ell'}$ for some integer $t_{m+1} \leq \max\{t_m,2\varPhi\} \leq \max\{\kappa',2\varPhi\}$. By induction, these integers $b'_j$ are defined for all $j \in \bb{N}$. Hence, the restriction of $\varphi$ to the convex set $\varGamma := \mathcal{K} \cup (\bigcup_{j=1}^{\infty} F^+_j(b'_j))$ is $\ell'$-periodic (but not $\ell$-periodic).

\begin{claim}
There is a translation $\mathcal{S}_1'$ of $\mathcal{S}$  contained in $\varGamma$ such that $\varphi \in \mathcal{M}_{0-}(\lle'\!,\mathcal{S}_1')$.	
\end{claim} 

Indeed, suppose not. Let $\mathcal{S}'$ be a translation of $\mathcal{S}$ such that the $(\lle'\!,\mathcal{S}'\!,q')$-half-strip $F^-(0)$ is contained in $\varGamma \backslash \mathcal{K}$ and consider $B' :=$ $F^-(0) \backslash F^-(2t')$, where  $t'\vec{v}_{\ell'}$ is a period of $\varphi| \sob{\mathcal{K}}$. Since $\mathcal{K}$ has a semi-infinite edge parallel to $\ell'$ and $\varGamma \backslash \mathcal{K} = \bigcup_{j=1}^{\infty} F^+_j(b'_j)$, then there exists an integer $\alpha' \geq 0$ such that both $F^-(\alpha')+i\vec{v}_{\ell}$ and $F^-(\alpha')+j\vec{v}_{\ell}$ are contained in $\varGamma \backslash \mathcal{K}$ for every integers $0 \leq i <$ $j \leq P_{\eta}(B')$. By the Pigeonhole Principle, $(T^{i\vec{v}_{\ell}}\varphi)|\sob{B'+\alpha'\vec{v}_{\ell'}} = (T^{j\vec{v}_{\ell}}\varphi)|\sob{B'+\alpha'\vec{v}_{\ell'}}$ for some integers $0 \leq i < j \leq$ $P_{\eta}(B')$. Hence, since $\varphi|\sob{\varGamma}$ is periodic of period  $t'\vec{v}_{\ell'}$, by the definition of $B'$,
\begin{equation}\label{eq_igualdade_Prop}
(T^{i\vec{v}_{\ell}}\varphi)|\sob{F^-(\alpha')} = (T^{j\vec{v}_{\ell}}\varphi)|\sob{F^-(\alpha')}.
\end{equation}
The contradiction hypothesis allows us  to apply Lemma \ref{rem_equival_lema} successively. So, from (\ref{eq_igualdade_Prop}), we obtain an unbounded region contained in $\varGamma$ that has nonempty intersection with $\varGamma \backslash \mathcal{K}$ and such that the restriction of $\varphi$ to it is $\ell$-periodic (see Figure~\ref{duas_faixa_afir}). 
\begin{figure}[ht]
	\centering
	\def\svgwidth{4.5cm}
	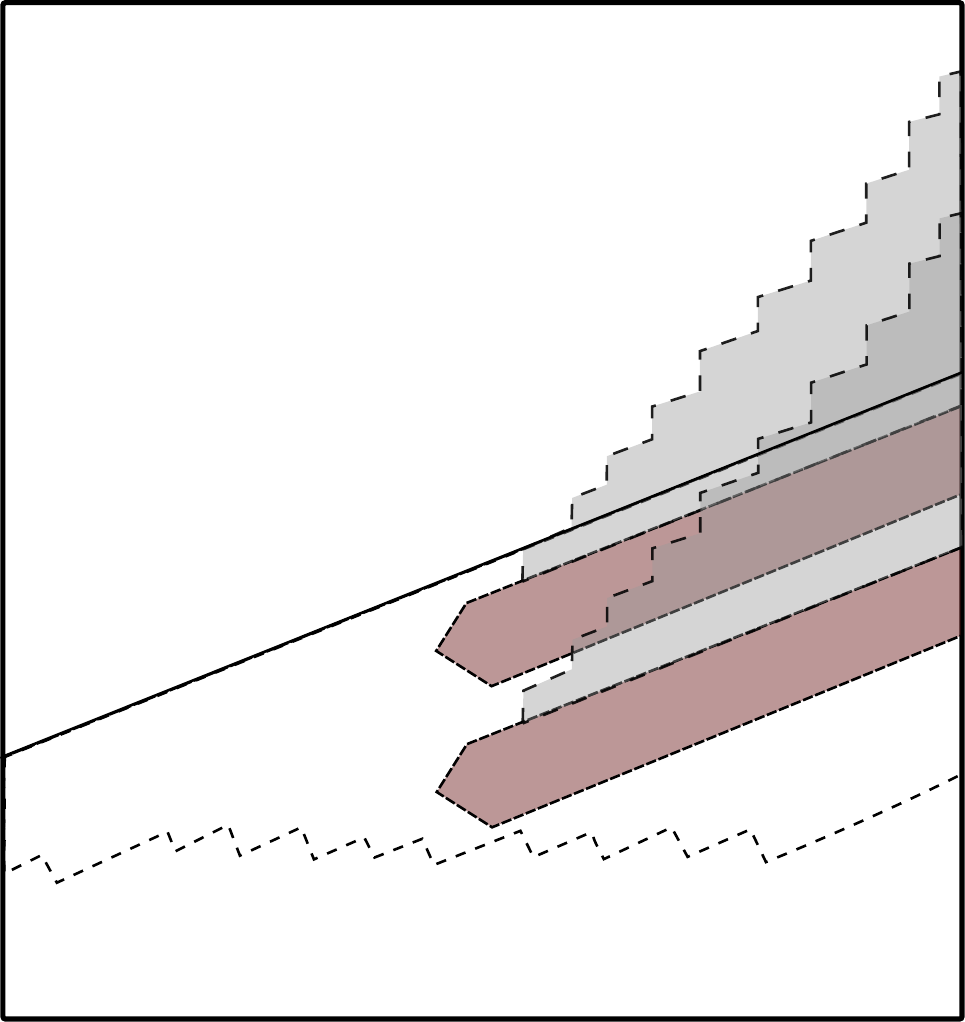
	\caption{The \hspace{3cm} $F^-(\alpha')+i \vec{v}_{\ell}$ and $F^-(\alpha')+j \vec{v}_{\ell}$.}
	\label{duas_faixa_afir}
	\vspace{-13pt}\hspace{-2.5cm}$(\lle'\!,\mathcal{S}''\!,q')$-half-strips
\end{figure}
This allows us to contradict the maximality of $\mathcal{K}$, since the set of doubly periodicity of $\varphi$ would now include the set $\{g+t\vec{v}_{\ell'} : t \geq 0\}$ for some $g \in \ell'{}^{(-)}_{\hspace{-0.1cm}\scriptscriptstyle\mathcal{K}} \cap \bb{Z}^2$, which proves the claim.

So, let $\mathcal{S}'_1$ be a translation of $\mathcal{S}$ contained in $\varGamma$ such that $\varphi \in \mathcal{M}_{0-}(\lle'\!,\mathcal{S}_1')$. Note that the inclusion $\mathcal{S}'_1 \subset \varGamma$ implies that the $(\lle'\!,\mathcal{S}'_1,q')$-half-strip $F^-_1(0)$ is contained in $\varGamma$, but obviously may not be contained in $\mathcal{K}$. Define $\mathcal{S}'_{j+1} := \mathcal{S}_1'-ju$ for all inte- ger $j \geq 1$ (recall that $\ell'+u = \ell'{}^{(-)}$). Denot\-ing $a'_1 := 0$, for $j > 1$, let $a'_j \in \bb{Z}$ be such that the $(\lle'\!,\mathcal{S}'_j,q')$-half-strip $F^-_j(a'_{j})$ lies in $(\lle'_{\scriptscriptstyle\mathcal{S}_{j-1}}\! \cup F^-_{j-1})(a'_{j-1}+q')$. Let $J \geq 1$ be the largest integer such that $\varphi \in \mathcal{M}_{a'_J \!-}(\lle'\!,\mathcal{S}'_J)$. Thanks to  Proposition~\ref{pps_cap3_baixa_compl}, the res- triction of $\varphi$ to $(\lle'_{\scriptscriptstyle\mathcal{S}_{J}}\! \cup F^-_J)(a'_{J}+q')$ is periodic of period $t'_0\vec{v}_{\ell'}$ for some integer  $t'_0 \leq$ $2\lceil \frac{1}{2}|\lle'_{\scriptscriptstyle\mathcal{S}} \cap \mathcal{S}| \rceil-2$. Thus, since $\varphi \not\in \mathcal{M}_{a'_j \!-}(\lle'\!,\mathcal{S}'_j)$ for all integer $j > J$, condition~(ii) of Proposition \ref{pps_cap3_baixa_compl} ensures that the restriction of $\varphi$ to $(\lle'_{\scriptscriptstyle\mathcal{S}_{j}}\! \cup F^-_j)(a'_j)$ is periodic of period $t'_0\vec{v}_{\ell'}$. This means that the restriction of $\varphi$ to $\varLambda := \bigcup_{j=J}^{\infty} F^-_j(a'_j+q')$ is periodic of period $t'_0\vec{v}_{\ell'}$, where $t'_0 \leq 2\lceil \frac{1}{2}|\lle'_{\scriptscriptstyle\mathcal{S}} \cap \mathcal{S}| \rceil-2 \leq$ $|\lle'_{\scriptscriptstyle\mathcal{S}} \cap \mathcal{S}|-1 \leq |\ell'_{\scriptscriptstyle\mathcal{S}} \cap \mathcal{S}|-2$.

Observe that $\varLambda \cap \mathcal{K}$ is an unbounded region that can be described making use of a union of half-lines parallel to $\ell'$. Hence, since $\varphi|\sob{\mathcal{K}}$ is $\ell'$-periodic, it is easy to argue using the Fine-Wilf Theorem that $\varphi|\sob{\mathcal{K}}$ is also periodic of period $t'_0\vec{v}_{\ell'}$. This concludes the proof of condition (i) in the statement of proposition.

We also prove condition (ii) by considering two cases separately. We begin by assuming that
\begin{equation}\label{eq_cap3_difr}
(T^{i\,\vec{v}_{\lle}}\varphi)|\sob{\mathcal{S}} \neq (T^{j\,\vec{v}_{\lle}+h'}\varphi)|\sob{\mathcal{S}} \quad \forall \ i,j \geq a,
\end{equation} 
where $h' = \kappa'\vec{v}_{\ell'}$ is a period of $\varphi|\sob{\mathcal{K}}$. Clearly,
\begin{equation}\label{eq_cap3_igual_}
(T^{t\vec{v}_{\lle}}\varphi)|\sob{\mathcal{S} \backslash \ell_{\scriptscriptstyle\mathcal{S}}} = (T^{t\vec{v}_{\lle}+h'}\varphi)|\sob{\mathcal{S} \backslash \ell_{\scriptscriptstyle\mathcal{S}}} \quad \forall \ t \geq a.
\end{equation} 
For $A := \{(T^{t\vec{v}_{\lle}}\varphi)|\sob{\mathcal{S}} : t \geq a\}$ and $B := \{(T^{t\vec{v}_{\lle}+h'}\varphi)|\sob{\mathcal{S}} : t \geq a\}$,  from (\ref{eq_cap3_difr}) and (\ref{eq_cap3_igual_}) one has $|A|+|B| = |A \cup B| \leq \sum_{\gamma \in C} N_{\mathcal{S}}(\ell,\gamma)$, where $C := \{(T^{t\vec{v}_{\lle}+h'}\varphi)|\sob{\mathcal{S} \backslash \ell_{\scriptscriptstyle\mathcal{S}}} : t \geq a\}$. In particular, as $|C| \leq |B|$, then $$|A| \leq \left(\sum_{\gamma \in C} N_{\mathcal{S}}(\ell,\gamma)\right)-|C| \leq \sum_{\gamma \in C} \Big(N_{\mathcal{S}}(\ell,\gamma)-1\Big) \leq  P_{\eta}(\mathcal{S})-P_{\eta}(\mathcal{S} \backslash \ell_{\scriptscriptstyle\mathcal{S}}).$$ Hence,  (\ref{eq_mlc-generating}) applied to $\mathcal{T} = \mathcal{S} \backslash \ell_{\scriptscriptstyle\mathcal{S}}$ provides $|A| \leq \lceil \frac{1}{2}|\ell_{\scriptscriptstyle\mathcal{S}} \cap \mathcal{S}| \rceil-1$. Let $\xi = (\xi_t)_{t \geq a}$ be the sequence defined by $\xi_t := (T^{t\vec{v}_{\lle}}\varphi)|\sob{\mathcal{F}^{\ell,p}(\mathcal{S}) \cup \{f_{\mathcal{S}}(\ell_{\scriptscriptstyle\mathcal{S}})\}}$ for all integer $t \geq a$ and consider $n := \lceil \frac{1}{2}|\ell_{\scriptscriptstyle\mathcal{S}} \cap \mathcal{S}| \rceil-1$. Since $n < |\ell_{\scriptscriptstyle\mathcal{S}} \cap \mathcal{S}|-1$ and $|\ell_{\scriptscriptstyle\mathcal{S}} \cap \mathcal{S}| \leq |\lle_{\scriptscriptstyle\mathcal{S}} \cap \mathcal{S}|$, from Remark~\ref{rem_cap2_arest_paral} we have that $P_{\xi}(n) \leq |A| \leq n$. As $2 \leq  n-P_{\xi}(1)+2 \leq n$ if $P_{\xi}(1) > 1$,  then, for $k = n-P_{\xi}(1)+2$, we have $P_{\xi}(k) \leq P_{\xi}(n) \leq n=k+P_{\xi}(1)-2$ and therefore, by Alphabetical Morse-Hedlund Theorem, that the sequence $(\xi_t)_{t \geq a+n}$ is periodic of period at most $n$. This means, even if $P_{\xi}(1) = 1$, that the restriction of $\varphi$ to the set $(\ell_{\scriptscriptstyle\mathcal{S}} \cup F^-)(a+n)$ is periodic of period $\tau_0\vec{v}_{\lle}$ for some $\tau_0 \leq \lceil \frac{1}{2}|\ell_{\scriptscriptstyle\mathcal{S}} \cap \mathcal{S}| \rceil-1$. Being $\varphi|\sob{(\ell_{\scriptscriptstyle\mathcal{S}} \cup F^-)(a)}$ $\ell$-periodic (see Lemma \ref{lem_extens_semifaixa}), the inclusion $$(\ell_{\scriptscriptstyle\mathcal{S}} \cup F^-)(a+n) \subset (\ell_{\scriptscriptstyle\mathcal{S}} \cup F^-)(a)$$ ensures that $\varphi|\sob{(\ell_{\scriptscriptstyle\mathcal{S}} \cup F^-)(a)}$ is also periodic of period $\tau_0\vec{v}_{\lle}$, which concludes this case. 

In the second case, we assume that there exist integers $i,j \geq a$ such that \begin{equation}\label{eq_cap3_main_prop}
\varphi|\sob{\mathcal{S}+i\vec{v}_{\lle}} = (T^{i\,\vec{v}_{\lle}}\varphi)|\sob{\mathcal{S}} = (T^{j\,\vec{v}_{\lle}+h'}\varphi)|\sob{\mathcal{S}} = (T^{h'}\varphi)|\sob{\mathcal{S}+j\vec{v}_{\lle}}.
\end{equation}
As guaranteed by condition (i), $\varphi|\sob{F^-(i)}$ and $(T^{h'}\,\varphi)|\sob{F^-(j)}$ are periodic of period $t_0\vec{v}_{\lle}$ for some $t_0 \leq \lceil \frac{1}{2}|\ell_{\scriptscriptstyle\mathcal{S}} \cap \mathcal{S}| \rceil-1$. Hence,  (\ref{eq_cap3_main_prop}) and Remark \ref{rem_cap2_arest_paral} allow us to get $$\varphi|\sob{F^-(i)\cup(\mathcal{S}+i\vec{v}_{\lle})} = (T^{h'}\!\varphi)|\sob{F^-(j)\cup(\mathcal{S}+j\vec{v}_{\lle})}.$$ Being $\mathcal{S}$ an $\eta$-generating set, by induction, it follows that the restrictions of $\varphi$ to the sets $(\ell_{\scriptscriptstyle\mathcal{S}} \cup F^-)(i)$ and $(\ell_{\scriptscriptstyle\mathcal{S}} \cup F^-)(j)+h' \subset$
$\mathcal{K}$ coincide. Thus, since the restriction of $\varphi$ to $(\ell_{\scriptscriptstyle\mathcal{S}} \cup F^-)(j)+h'$ is periodic of period $t_0\vec{v}_{\lle}$, then $\varphi|\sob{(\ell_{\scriptscriptstyle\mathcal{S}} \cup F^-)(i)}$ and therefore $\varphi|\sob{(\ell_{\scriptscriptstyle\mathcal{S}} \cup F^-)(a)}$ is periodic of period $t_0\vec{v}_{\lle}$.
\end{proof}

\section{Proof of the main result}
\label{sec7}

Supposing, by contradiction, that the thesis of Theorem \ref{main_result} does not hold, following ideas highlighted by Cyr and Kra, we will reach an absurd from the existence of an $(\ell,\ell')$-periodic maximal $\mathcal{K}$-configuration. The main idea is to argue that the number of configurations arising in the boundary transition from doubly periodicity to aperiodicity would be greater than possible.

\medbreak\noindent{\it Proof of Theorem \ref{main_result}}\hspace{1ex}\ignorespaces Suppose, by contradiction, that $\eta$ is aperiodic. Let $\mathcal{S} \in \mathcal{F}^{V\!ol}_{C}$ be an mlc $\eta$-generating set. Corollary~\ref{cor_2_periodic} ensures that there is at least one nonexpansive line on $X_{\eta}$, which is denoted by $\ell \in \bb{G}_1$. As observed in Remark~\ref{obs_exist_arestas_mcbc}, the antiparallel lines $\ell$ and $\lle$ are both one-sided nonexpansive directions on $X_{\eta}$. This allows us to assume, without loss of generality, $|\ell^{}_{\scriptscriptstyle\mathcal{S}} \cap \mathcal{S}| \leq |\lle^{}_{\scriptscriptstyle\mathcal{S}} \cap \mathcal{S}|$. Hence, according to Proposition \ref{pps_exist_maximal_conf}, there exists an $(\ell,\ell')$-periodic maximal $\mathcal{K}$-configuration $\varphi \in X_{\eta}$. Recall that $\ell' \in$ $\bb{G}_1$ is a one-sided nonexpansive directions on $\overline{Orb \, (\varphi)}$ and so  also on $X_{\eta}$ which satisfies $\vec{v}_{\ell'} \in \mathcal{H}(\ell)$. We denote by $\psi \in A^{\bb{Z}^2}$ the unique doubly periodic configuration such that $\varphi|\sob{\mathcal{K}} = \psi|\sob{\mathcal{K}}$. Translating $\mathcal{S}$, we can assume that the $(\ell,\mathcal{S},p)$-half-strip $F^-(0)$ lies in $\mathcal{K}$ and that $\ell^{}_{\scriptscriptstyle\mathcal{S}} = \ell_{\scriptscriptstyle\mathcal{K}}^{(-)}$, where $p := |\ell_{\scriptscriptstyle\mathcal{S}} \cap \mathcal{S}|-1$. As Proposition~\ref{pps_lp-balanceado} guarantees the existence of a balanced set with respect to every gi\-ven one-sided nonexpansive direction, by Proposition \ref{afm_principal}, $\varphi|\sob{\mathcal{K}}$ and so $\psi$ are doubly periodic configurations of periods $h := t_0\vec{v}_{\lle}$ and $h' := t'_0\vec{v}_{\ell'}$, where
\begin{equation}\label{eq_cap3_lim1}
t_0 \leq \Big\lceil \frac{1}{2}|\ell_{\scriptscriptstyle\mathcal{S}} \cap \mathcal{S}| \Big\rceil-1 \leq |\ell_{\scriptscriptstyle\mathcal{S}} \cap \mathcal{S}|-2, \quad t'_0 \leq |\ell'_{\scriptscriptstyle\mathcal{S}} \cap \mathcal{S}|-2.
\end{equation}

For any rational oriented line $l \in \bb{G}_1$, let $l_0 := l_{\scriptscriptstyle\mathcal{S}}$ and denote $l_{i+1} := l^{(+)}_i$ for all $i \geq 1$. For suitable integers $d \geq 1$ and $q \geq 0$, we define $$\mathcal{S}_{l}(d,q) := \bigcup_{i=0}^{d-1} \left\{f_{\mathcal{S}}(l_i)-t\vec{v}_{l} : 0 \leq t \leq q\right\},$$ where $f_{\mathcal{S}}(l_i) \in \bb{Z}^2$ is the final point of $l_i \cap \mathcal{S}$ (with respect to the orientation of $l_i$). In order to avoid a heavy notation, we consider $\mathcal{S}(d) :=$ $\mathcal{S}_{\lle\,}(d,p-2)$ for each integer $1 \leq d \leq \diam_{\ell}(\mathcal{S})$ and $\mathcal{S}^* :=$ $\mathcal{S} \backslash \mathcal{S}_{\ell}(\diam_{\ell}(\mathcal{S}),0)$ (See Figure \ref{defdem}). Recall that, thanks to Lemma \ref{lem_cardinal_mlc}, $p \geq 2$. Since $\mathcal{S}^*$ is a proper convex subset of $\mathcal{S}$ and $\mathcal{S}$ is an mlc $\eta$-ge\-nerating set, by (\ref{eq_mlc-generating}) one has
\begin{equation}\label{eq_cap3_des_contra}
P_{\eta}(\mathcal{S})-P_{\eta}(\mathcal{S}^*) \leq \Big\lceil\frac{1}{2}|\mathcal{S} \backslash \mathcal{S}^*| \Big\rceil -1 < |\mathcal{S} \backslash \mathcal{S}^*|-1.
\end{equation}

\begin{figure}[ht]
	\centering
	\def\svgwidth{9.4cm}
	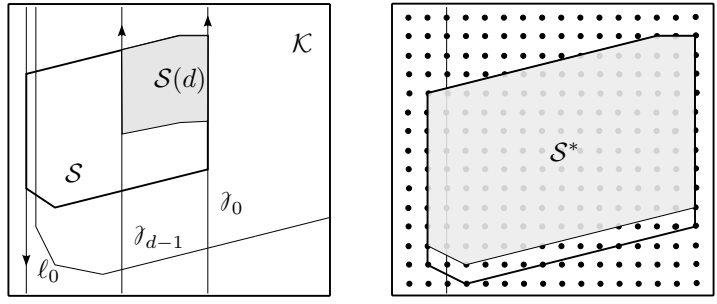
	\caption{Representation of the sets $\mathcal{S}(d)$ and $\mathcal{S}^*$.}
	\label{defdem}
\end{figure}

Let $\wp_0 \subset \bb{R}^2$ be the oriented line parallel to $\ell$ such that  $\wp_0 \cap (\mathcal{S} \backslash \mathcal{S}^*)$ $= i_{\mathcal{S}}(\ell'_{\scriptscriptstyle\mathcal{S}})$ and denote $\wp_{i+1} = \wp^{(+)}_i$ for all $i \geq 0$. If $m' \geq 1$ is such that $\wp_{m'} \cap (\mathcal{S} \backslash \mathcal{S}^*) = f_{\mathcal{S}}(\ell'_{\scriptscriptstyle\mathcal{S}})$, define $z_i := \wp_i \cap (\mathcal{S} \backslash \mathcal{S}^*)$ for each $0 \leq i \leq m'$. By (\ref{eq_cap3_lim1}) there exist integers $0 \leq i < j \leq m'$ such that $t'_0\vec{v}_{\ell'} = z_j-z_i$. Thus, we define $$d' := \min\left\{j-i : 0 \leq i < j \leq m', \ \psi \ \textrm{is periodic of period} \ z_{j}-z_i\right\}.$$ Identifying the vector $z_j-z_i$ to an oriented line segment $S_{ij} \subset \bb{R}^2$, note that $j-i+1$ represents the number of distinct oriented lines $\ell'' \subset \bb{R}^2$ parallel to $\ell$ such that $\ell'' \cap \bb{Z}^2 \neq \emptyset$ and $\ell'' \cap S_{ij} \neq \emptyset$. Then, writing $\diam_{\ell}(\mathcal{S}) = (|\ell'_{\scriptscriptstyle\mathcal{S}} \cap \mathcal{S}|-1)\mu+1+r$, where $\mu := j-i$ for any $0 \leq i < j \leq m'$ such that $\vec{v}_{\ell'} = z_j-z_i$ and $r \geq 0$, by (\ref{eq_cap3_lim1}) one has $$\diam_{\ell}(\mathcal{S}) > (|\ell'_{\scriptscriptstyle\mathcal{S}} \cap \mathcal{S}|-2)\mu+1+r \geq t'_0\mu+1,$$ which yields $d' \leq  t'_0\mu < \diam_{\ell}(\mathcal{S})-1$. In particular, $\ell_{\scriptscriptstyle\mathcal{S}} \cap \mathcal{S}(d') = \emptyset$. Defining $\bar{d} :=$ $|\mathcal{S} \backslash \mathcal{S}^*|-d' = \diam_{\ell}(\mathcal{S})-d' > 1$, it is clear that $(\mathcal{S}(d')+lu)+i\vec{v}_{\lle} \subset \mathcal{K}$ for every $i \geq 0$ and $0 \leq l \leq \bar{d}-1$.

Let $u \in \bb{Z}^2$ be such that $\ell+u = \ell^{(-)}$. For each integer $l \geq 0$, we define $\mathcal{S}_l :=$ $\mathcal{S}+lu$ and $\mathcal{S}^*_l := \mathcal{S}^*+lu$. Since $\varphi \in X_{\eta}$ is aperiodic and $\ell \in \bb{G}_1$ is a one-sided  nonexpansive direction on $\overline{Orb \, (\varphi)}$, from Proposition \ref{pps_lp-balanceado}, Corollary~\ref{cor_reta_antiparalela} and Proposition \ref{pps_period_global}, we conclude that the restriction of $\varphi$ to every $(\ell,\mathcal{S}_l,p)$-strip is not $\ell$-periodic. For each integer $0 \leq l \leq \bar{d}-1$, let $a_l \in \bb{Z}$ be the smallest integer for which the restriction of $\varphi$ to the set $(\ell_{\scriptscriptstyle\mathcal{S}_l} \cup F^-)(a_l)$ is $\ell$-periodic, where $F^-(a_l)$ denotes as usual the corresponding $(\ell,\mathcal{S}_l,p)$-half-strip. We remark that the existence of $a_l$ follows from Lemma~\ref{lem_extens_semifaixa} applied successively in order to get a larger region that contains $\mathcal{K}$ and such that the restriction of $\varphi$ is $\ell$-periodic (but not doubly periodic).

In the next two claims, we will count the number of configurations that arise in the boundary transition from doubly periodicity to aperiodicity. The strategy is to show that there are so many such configurations that (\ref{eq_cap3_des_contra}) may be contradicted.

\begin{claim}\label{claim_lad}
The following conditions hold.
\begin{enumerate}
	\item[(i)] $\big|\{\varphi|\sob{\mathcal{S}_l^*+(a_{l}-1)\vec{v}_{\lle}} : 0 \leq l \leq \bar{d}-1\}\big| = \bar{d}$.
	\item[(ii)] For each integer $0 \leq l \leq \bar{d}-1$, there exist at least two distinct $\mathcal{S}_l$-configura- tions $\gamma',\gamma'' \in L(\mathcal{S}_l,\eta)$ such that $\gamma'|\sob{\mathcal{S}^*_l} = (T^{(a_{l}-1)\vec{v}_{\lle}}\varphi)|\sob{\mathcal{S}_l^*} = \gamma''|\sob{\mathcal{S}^*_l}$.
\end{enumerate}
\end{claim}

Indeed, with respect to condition (i), suppose by contradiction, that there exist integers $0 \leq l < l' \leq \bar{d}-1$ such that
\begin{equation*}
\varphi|\sob{\mathcal{S}_l^*+(a_l-1)\vec{v}_{\lle}} = \varphi|\sob{\mathcal{S}_{l'}^*+(a_{l'}-1)\vec{v}_{\lle}}.
\end{equation*}
Let $\hat{\mathcal{S}}^* := \mathcal{S}^* \backslash \mathcal{S}_{\ell}(l',\hat{q})$, where $\hat{q} := \max\{|\ell_i \cap \mathcal{S}| : 0 \leq i < \diam_{\ell}(\mathcal{S})\}$. Note that $\hat{\mathcal{S}}^*+l'u \subset \mathcal{H}(\ell_{\scriptscriptstyle\mathcal{S}})$ and $(\hat{\mathcal{S}}^*+l 'u) \cap \ell_{\scriptscriptstyle\mathcal{S}} \neq \emptyset$. If $t_{l}\vec{v}_{\lle}$ and $t_{l'}\vec{v}_{\lle}$ are, respectively, periods of $\varphi|\sob{(\ell_{\scriptscriptstyle\mathcal{S}_l}\! \cup F^-)(a_l)}$ and $\varphi|\sob{(\ell_{\scriptscriptstyle\mathcal{S}_{l'}}\! \cup F^-)(a_{l'})}$, choose an integer $m \geq 1$ such that, for $J := lu+(a_l-1+mt_{l})\vec{v}_{\lle}$ and $J' := l'u+(a_{l'}-1+mt_{l'})\vec{v}_{\lle}$, both $\hat{\mathcal{S}}^*+J$ and $\hat{\mathcal{S}}^*+J'$ are contained in $(\ell_{\scriptscriptstyle\mathcal{S}} \cup F^-)(0)$. In particular,
\begin{equation}\label{eq_cap3_contraexm1}
\varphi|\sob{\hat{\mathcal{S}}^*+J} = \varphi|\sob{\hat{\mathcal{S}}^*+lu+(a_l-1)\vec{v}_{\lle}} = \varphi|\sob{\hat{\mathcal{S}}^*+l'u+(a_{l'}-1)\vec{v}_{\lle}} = \varphi|\sob{\hat{\mathcal{S}}^*+J'}.
\end{equation}
Since any line parallel to $\ell$ that has nonempty intersection with $\mathcal{S}$ contains at least $p$ points (see Remark \ref{rem_cap2_arest_paral}), then every line parallel to $\ell$ that has nonempty intersection with $\hat{\mathcal{S}}^*$ contains at least $p-1 \geq 1$ points. For $$\mathcal{F}_I := \Big(\mathcal{F}^{\ell,p-1}(\hat{\mathcal{S}}^*) \cup \{f_{\hat{\mathcal{S}}^*}(\ell_{\scriptscriptstyle\hat{\mathcal{S}}^*})\}\Big)+I, \quad \textrm{with} \ I \in \{J,J'\},$$ define the sequences $\xi = (\xi_t)_{t \geq 0}$ and $\xi' = (\xi'_t)_{t \geq 0}$ by $\xi_t :=  (T^{t\vec{v}_{\lle}}\varphi)|\sob{\mathcal{F}_J}$ and $\xi'_t := (T^{t\vec{v}_{\lle}}\varphi)|\sob{\mathcal{F}_{J'}}$ for all $t \geq 0$. It follows from condition (ii) of Proposition \ref{afm_principal} that $\xi$ and $\xi'$ are periodic of periods $n,n' \leq \lceil \frac{1}{2}|\ell_{\scriptscriptstyle\mathcal{S}} \cap \mathcal{S}| \rceil-1$. Since $\xi$ and $\xi'$ coincide in at least $p-1$ consecutive indexes (see (\ref{eq_cap3_contraexm1})) and $n+n'-\gcd(n,n') \leq |\ell_{\scriptscriptstyle\mathcal{S}} \cap \mathcal{S}|-2 = p-2$, Fine-Wilf Theorem guarantees that $\xi = \xi'$. In particular, for $\hat{F}^* := \bigcup_{t \geq 0} (\hat{\mathcal{S}}^*+t\vec{v}_{\lle})$, we have $\varphi|\sob{\hat{F}^*+J} = \varphi|\sob{\hat{F}^*+J'}$. Assuming $d' = j-i$, recall that $\varphi|\sob{\mathcal{K}}$ is periodic of period $v := z_j-z_i$. Then, if $g' = g+J' \in (\hat{F}^*+J') \cap \ell_{\scriptscriptstyle\mathcal{S}}$, as $g'+\vec{v}_{\lle} \in (\hat{F}^*+J') \cap \ell_{\scriptscriptstyle\mathcal{S}}$ and $\varphi|\sob{\hat{F}^*+J} = \varphi|\sob{\hat{F}^*+J'}$, we obtain that $$\varphi_{g'+\vec{v}_{\lle}} = \varphi_{g+\vec{v}_{\lle}+J'} = \varphi_{g+\vec{v}_{\lle}+J} = \varphi_{g+\vec{v}_{\lle}+J+v}.$$ As $g \in \hat{F}^* \cap \ell_{l'}$ and $\diam_{\ell}(\hat{\mathcal{S}}^*) = \diam_{\ell}(\mathcal{S})-l' \geq d'+1$, we have $g+\vec{v}_{\lle}+v \in$ $\hat{F}^*$ (see Figure \ref{afirm}) and therefore that $g+\vec{v}_{\lle}+v+J \in (\hat{F}^*+J)$. Therefore, $$\varphi_{g+\vec{v}_{\lle}+v+J} = \varphi_{g+\vec{v}_{\lle}+v+J'} = \varphi_{g'+\vec{v}_{\lle}+v}.$$ We have then $\varphi_{g'+\vec{v}_{\lle}} = \varphi_{g'+\vec{v}_{\lle}+v}$. Thus, the restriction of $\varphi$ to the set $\{g'+\vec{v}_{\lle} : g' \in$ $(\hat{F}^*+J') \cap \ell_{\scriptscriptstyle\mathcal{S}}\} \cup \mathcal{K}$ is doubly periodic of periods $v =$ $z_j-z_i$ and $\kappa \vec{v}_{\lle}$ for some $\kappa \in \bb{N}$. So there exists a convex set that strictly contains $\mathcal{K}$ and such that the restriction of $\varphi$ to it is doubly periodic with periods contained in $\ell \cap \bb{Z}^2$ and $\ell' \cap \bb{Z}^2$. But this contradicts the maximality of $\mathcal{K}$.
\begin{figure}[ht]
	\centering
	\def\svgwidth{4.7cm}
	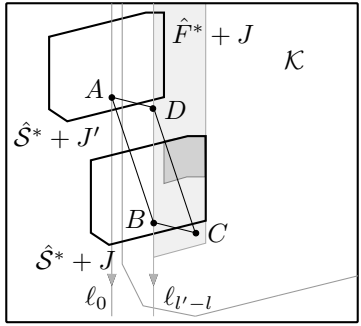
	\caption{The equality $\varphi_A = \varphi_B = \varphi_C = \varphi_D$, where $A = $ $g-\vec{v}_{\ell}+J', B = g-\vec{v}_{\ell}+J, C = B+v$ and $D = A+v$.}
	\label{afirm}
\end{figure}

With respect to condition (ii), since $\mathcal{S}_l^*+(a_l-1)\vec{v}_{\lle} \subset (\ell_{\scriptscriptstyle\mathcal{S}_l}\! \cup F^-)(a_l)$, from the minimality of $a_{l}$ follows the statement, which proves the claim.\medbreak

For each integer $0 \leq l \leq d'-1$, let $u_l$ denote the vector $z_{m'}-z_{m'-d'+l}$. Let $\mathcal{T}$ be a translation of $\mathcal{S}$ such that $\ell'_{\scriptscriptstyle\mathcal{T}} = \ell'{}_{\scriptscriptstyle\hspace{-0.08cm}\mathcal{K}}^{(-)}$ and $\mathcal{T} \backslash \ell'_{\scriptscriptstyle\mathcal{T}} \subset \mathcal{K}$, and let $\mathcal{T}^*$ be the corresponding translation of $\mathcal{S}^*$. Since $z_{i}$ belongs to $\mathcal{S} \backslash \mathcal{S}^*$ and does not belong to $\ell'_{\scriptscriptstyle\mathcal{S}^*}$ for every $0 \leq i \leq m'$, it is not difficult to check that $\mathcal{T}^*+u_l \subset \mathcal{K}$ for every integer $0 \leq$ $l \leq d'-1$ (even when $z_{m'-d'+l} \in \ell'{}_{\scriptscriptstyle\hspace{-0.1cm}\mathcal{S}^*}^{(-)}$).

\begin{claim}\label{claim_emb}
The following conditions hold.
\begin{enumerate}
	\item[(iii)] $\big|\{\varphi|\sob{\mathcal{T}^*+u_l} : 0 \leq l \leq d'-1\}\big| = d'$.
	\item[(iv)] For each integer $0 \leq l \leq d'-1$, there exist at least two distinct $\mathcal{T}$-configura- tions $\gamma',\gamma'' \in L(\mathcal{T},\eta)$ such that $\gamma'|\sob{\mathcal{T}^*} = (T^{u_l}\varphi)|\sob{\mathcal{T}^*} = \gamma''|\sob{\mathcal{T}^*}$.
\end{enumerate}
\end{claim}

Indeed, with respect to condition (iii), suppose, by contradiction, that there are integers $0 \leq l < l' \leq d'-1$ such that
\begin{equation}\label{eq_cap3_pnul}
\varphi|\sob{\mathcal{T}^*+u_l} = \varphi|\sob{\mathcal{T}^*+u_{l'}}.
\end{equation} 
As remarked above, both $\mathcal{T}^*+u_l$ and $\mathcal{T}^*+u_{l'}$ are contained in $\mathcal{K}$. Since the restric\-tions of $\varphi$ and $\psi$ to the set $\mathcal{K}$ coincide, by (\ref{eq_cap3_pnul}) and from the fact of $\psi$ being periodic of period $t_0\vec{v}_{\lle}$ (see (\ref{eq_cap3_lim1})), it follows that
\begin{equation}\label{eq_cap3_afirma311}
\psi|\sob{\tilde{F}+u_l} = \psi|\sob{\tilde{F}+u_{l'}},
\end{equation}
where $\tilde{F} := \bigcup_{i \in \bb{Z}} (\mathcal{T}+i\vec{v}_{\lle})$. Thus, as $\psi$ is periodic of period $t'_0\vec{v}_{\ell'}$, from (\ref{eq_cap3_lim1}) and (\ref{eq_cap3_afirma311}), it is easy to argue that $\psi$ is also periodic of period $u_{l}-u_{l'} = z_{m'-d'+l'}-z_{m'-d'+l}$. But $m'-d'+l'-(m'-d'+l) = l'-l \leq d'-1$, which contradicts the minimality of $d'$.

With respect to condition (iv), suppose by contradiction that, for some $\mathcal{T}^*$-con\-figuration $\gamma = (T^{u_l}\varphi)|\sob{\mathcal{T}^*}$, with $0 \leq l \leq d'-1$, there exists a unique $\mathcal{T}$-configuration $\gamma' \in L(\mathcal{T},\eta)$ such that $\gamma'|\sob{\mathcal{T}^*} = \gamma$. We first show that the restriction of $\varphi$ to the set $\mathcal{K} \cup \{g_0+j\vec{v}_{\ell'} : j \geq 0\}$, for some $g_0 \in \ell'{}_{\scriptscriptstyle\hspace{-0.08cm}\mathcal{K}}^{(-)} \cap$ $(\mathcal{T}+u_l)$, is periodic of period $t'_0\vec{v}_{\ell'}$. Consider the final point $g_0 := f_{\mathcal{T}+u_l}(\ell'{}_{\scriptscriptstyle\hspace{-0.08cm}\mathcal{K}}^{(-)})$ and, for all $j \geq 0$, define $g_j := g_0+j\vec{v}_{\ell'}$. Let $\hat{\mathcal{T}}$ be the translation of $\mathcal{T}+u_l$ such that $g_1$ is the final point of $\ell'{}_{\scriptscriptstyle\hspace{-0.08cm}\hat{\mathcal{T}}}^{(-)} \cap \hat{\mathcal{T}}$. Note that the parallelogram of vertices $i_{\mathcal{T}}(\ell'_{\scriptscriptstyle\mathcal{T}}),$ $i_{\mathcal{T}}(\ell'_{\scriptscriptstyle\mathcal{T}})+\vec{v}_{\lle}, f_{\mathcal{T}}(\ell'_{\scriptscriptstyle\mathcal{T}}), f_{\mathcal{T}}(\ell'_{\scriptscriptstyle\mathcal{T}})+\vec{v}_{\lle}$ lies in $\mathcal{T}$ and, therefore, by Remark \ref{rem_cap2_arest_paral} 
\begin{equation}\label{eq_theor_almost_prin}
|\ell'{}_{\scriptscriptstyle\hspace{-0.08cm}\mathcal{K}}^{(-)} \cap (\mathcal{T}+u_l)| \geq |\ell'_{\scriptscriptstyle\mathcal{T}} \cap \mathcal{T}|-1 = |\ell'_{\scriptscriptstyle\hat{\mathcal{T}}}\! \cap \hat{\mathcal{T}}|-1.
\end{equation}
Since $\gamma = \varphi|\sob{\mathcal{T}^*+u_l} = \varphi|\sob{\mathcal{T}^*+u_l+t'_0\vec{v}_{\ell'}}$, then $\gamma'= \varphi|\sob{\mathcal{T}+u_l} = \varphi|\sob{\mathcal{T}+u_l+t'_0\vec{v}_{\ell'}}$. By (\ref{eq_theor_almost_prin}), one has $$\ell'{}_{\scriptscriptstyle\hspace{-0.08cm}\mathcal{K}}^{(-)} \cap \hat{\mathcal{T}} \backslash \{g_1\} \subset \ell'{}_{\scriptscriptstyle\hspace{-0.08cm}\mathcal{K}}^{(-)} \cap (\mathcal{T}+u_l)$$ and $$\ell'{}_{\scriptscriptstyle\hspace{-0.08cm}\mathcal{K}}^{(-)} \cap (\hat{\mathcal{T}}+t'_0\vec{v}_{\ell'}) \backslash \{g_{t'_0+1}\} \subset \ell'{}_{\scriptscriptstyle\hspace{-0.08cm}\mathcal{K}}^{(-)} \cap (\mathcal{T}+u_l+t_0\vec{v}_{\ell'}).$$ Therefore, $\varphi|\sob{\hat{\mathcal{T}} \backslash \{g_1\}} = \varphi|\sob{(\hat{\mathcal{T}}+t'_0\vec{v}_{\ell'}) \backslash \{g_{t'_0+1}\}}$. As $\hat{\mathcal{T}}$ is $\eta$-generating, we obtain $g_1 =$ $g_{t'_0+1} =$ $g_1+t'_0\vec{v}_{\ell'}$. By induction, it follows that $\varphi$ restrict to $\mathcal{K} \cup \{g_0+t\vec{v}_{\ell'} : t \geq 0\}$ is periodic of period $t'_0\vec{v}_{\ell'}$. 

Since $\gamma = (T^{u_l+t_0\vec{v}_{\lle}}\varphi)|\sob{\mathcal{T}^*}$, reasoning in a similar way as above, we conclude that the restriction of $\varphi$ to $\mathcal{K} \cup \{g_0+t\vec{v}_{\ell'} : t \geq 0\}$ is also periodic of period $t_0\vec{v}_{\lle}$. This contradicts the maximality of $\mathcal{K}$, which proves the claim.\medbreak

To reach a contradiction and thus to conclude the proof, it is enough to show that the $\mathcal{S}^*$-configurations of condition (i) are different from  $\mathcal{S}^*$-configurations of con\-dition (iii). As a matter of fact, if this is the case, there exist at least $\bar{d}+d' = |\mathcal{S} \backslash \mathcal{S}^*|$ distinct $\mathcal{S}^*$-configurations $\gamma \in L(\mathcal{S}^*\!,\eta)$ such that $|\{\gamma' \in L(\mathcal{S},\eta) : \gamma'|\sob{\mathcal{S}^*} = \gamma\}| > 1$, which means that $$P_{\eta}(\mathcal{S}) - P_{\eta}(\mathcal{S}^*) \geq \bar{d}+d' = |\mathcal{S} \backslash \mathcal{S}^*|,$$ contradicting (\ref{eq_cap3_des_contra}). Note then that, since every configuration of condition (iii) be\-longs to $L(\mathcal{S}^*\!,\psi)$, it is enough to show that, for each integer $0 \leq l \leq \bar{d}-1$, $$(T^{lu+i\,\vec{v}_{\lle}}\varphi)|\sob{\mathcal{S}^*} \not \in L(\mathcal{S}^*\!,\psi), \quad \forall \ i \geq 0,$$ where $\ell+u = \ell^{\,(-1)}$. Suppose, by contradiction, that there exist $0 \leq l \leq \bar{d}-1$ and $i \geq 0$ such that
\begin{equation}\label{eq_cap3_afirma_2f}
(T^{lu+i\,\vec{v}_{\lle}}\varphi)|\sob{\mathcal{S}^*} \in L(\mathcal{S}^*\!,\psi).
\end{equation}
Since $\psi$ is doubly periodic of periods $t_0\vec{v}_{\lle}$ and $z_j-z_i$, with $0 \leq i < j \leq m'$ such that $d' = j-i$, the very definition of $\mathcal{S}(d') \subset \mathcal{S}^*$ ensures
\begin{equation}\label{eq_cap3_ext_unica}
\forall \ \gamma',\gamma'' \in L(\mathcal{S}^*\!,\psi), \quad \gamma'|\sob{\mathcal{S}(d')} = \gamma''|\sob{\mathcal{S}(d')} \ \ \textrm{implies} \ \ \gamma' = \gamma''. 
\end{equation}
As $(\mathcal{S}(d')+lu)+i\vec{v}_{\lle} \subset \mathcal{K}$ and $\varphi|\sob{\mathcal{K}} = \psi|\sob{\mathcal{K}}$, then $(T^{lu+i\,\vec{v}_{\lle}}\varphi)|\sob{\mathcal{S}(d')} \in$ $L(\mathcal{S}(d'),\psi)$. Since $h' = t'_0\vec{v}_{\ell'}$ is a period of $\varphi$, $(T^{lu+i\,\vec{v}_{\lle}}\varphi)|\sob{\mathcal{S}(d')} =  (T^{lu+i\,\vec{v}_{\lle}+mh'}\varphi)|\sob{\mathcal{S}(d')}$ for all $m \geq 1$. Note that, for $m$ sufficiently large, we have $(\mathcal{S}^*+lu)+i\vec{v}_{\lle}+mh' \subset \mathcal{K}$ and therefore $(T^{lu+i\,\vec{v}_{\lle}+mh'}\varphi)|\sob{\mathcal{S}^*} \in L(\mathcal{S}^*\!,\psi)$. Thus, by (\ref{eq_cap3_afirma_2f}) and (\ref{eq_cap3_ext_unica}) it follows that $$(T^{lu+i\,\vec{v}_{\lle}}\varphi)|\sob{\mathcal{S}^*} = (T^{lu+i\,\vec{v}_{\lle}+mh'}\varphi)|\sob{\mathcal{S}^*}.$$ For the half line $A := \ell_{\scriptscriptstyle\mathcal{S}} \cap (\ell_{\scriptscriptstyle\mathcal{S}} \cup F^-)(0)$, the above equality implies, in particular, that the restrictions given by $\varphi|\sob{A}$ and $(T^{mh'}\varphi)|\sob{A}$ coincide in at least $|\ell_{\scriptscriptstyle\mathcal{S}} \cap \mathcal{S}|-2$ consecutive indexes. Since $\varphi|\sob{A}$ and $(T^{mh'}\varphi)|\sob{A}$ are, respectively, periodic of periods $k\vec{v}_{\lle}$ and $t_0\vec{v}_{\lle}$, with $k,t_0 \leq \lceil \frac{1}{2}|\ell_{\scriptscriptstyle\mathcal{S}} \cap \mathcal{S}| \rceil-1$ (condition (ii) of Proposition~\ref{afm_principal} and (\ref{eq_cap3_lim1})), then, as $k+t_0-\gcd(k,t_0) \leq |\ell_{\scriptscriptstyle\mathcal{S}} \cap \mathcal{S}|-2$, from Fine-Wilf Theorem, we obtain that $\varphi|\sob{A} =$ $(T^{mh'}\varphi)|\sob{A}$, which contradicts the maximality of $\mathcal{K}$ and concludes the proof of the theorem.\medbreak

\acknowledgment{C. F. Colle would like to thank the Math-Am-Sud Project DCS 17-Math-01 for its support and the  LAMFA for its hospitality during the pre- paration of this manuscript.}

\end{document}